\documentclass{amsart}
\usepackage{hyperref}
\usepackage[english]{babel}
\usepackage{amsmath, amssymb, amsthm}

\newtheorem{theorem}{Theorem}[section]
\newtheorem{proposition}[theorem]{\bf Proposition}
\newtheorem{lemma}[theorem]{\bf Lemma}
\newtheorem{corollary}[theorem]{\bf Corollary}

\theoremstyle{definition}

\newtheorem{definition}[theorem]{\bf Definition}
\newtheorem{remark}[theorem]{\bf Remark}

\newcommand{\N}{\mathbb N}
\newcommand{\Z}{\mathbb Z}
\newcommand{\R}{\mathbb R}
\newcommand{\Q}{\mathbb Q}

\renewcommand{\P}{\mathbb P}

 \DeclareMathOperator{\ord}{ord}

 \DeclareMathOperator{\supp}{supp}

\renewcommand{\t}{\, | \,}

\numberwithin{equation}{section}

\begin{document}

\title[Sets of lengths are (almost) arithmetical progressions]{A characterization of  Krull monoids for which sets of lengths are (almost) arithmetical progressions}

\author[A. Geroldinger and W.A. Schmid]{Alfred Geroldinger and Wolfgang Alexander Schmid}

\begin{abstract}
Let $H$ be a Krull monoid with finite class group $G$ and suppose that every class contains a prime divisor. Then sets of lengths in $H$ have a well-defined structure which depends only on the class group $G$. With methods from additive combinatorics we establish a characterization of those class groups $G$ guaranteeing that all sets of lengths are (almost) arithmetical progressions.
\end{abstract}

\address[A.G.]{Institute for Mathematics and Scientific Computing\\ University of Graz, NAWI Graz\\ Heinrichstra{\ss}e 36\\ 8010 Graz, Austria} \email{alfred.geroldinger@uni-graz.at}

\address[W.S.]{Universit{\'e} Paris 13 \\ Sorbonne Paris Cit{\'e} \\ LAGA, CNRS, UMR 7539,  Universit{\'e} Paris 8\\ F-93430, Villetaneuse, France \\ and \\ Laboratoire Analyse, G{\'e}om{\'e}trie et Applications (LAGA, UMR 7539) \\ COMUE  Universit{\'e} Paris Lumi{\`e}res \\  Universit{\'e} Paris 8, CNRS \\  93526 Saint-Denis cedex, France}
\email{schmid@math.univ-paris13.fr}


\thanks{This work was supported by the Austrian Science Fund FWF, Project Number P28864-N35.}


\maketitle

\section{Introduction} \label{1}

Let $H$ be a Krull monoid with class group $G$ and suppose that every class contains a prime divisor. Then every element has a factorization into irreducibles. If $a = u_1 \cdot \ldots \cdot u_k$ with irreducibles $u_1, \ldots, u_k \in H$, then $k$ is called the factorization length. The set $\mathsf L (a) \subset \N_0$ of all possible factorization lengths is finite and called the set of lengths of $a$. The system $\mathcal L (H) = \{ \mathsf L (a) \mid a \in H \}$ is a well-studied means for describing the arithmetic of $H$. It is classic that $|L|=1$ for all $L \in \mathcal L (H)$ if and only if $|G| \le 2$ and if $|G| \ge 3$, then there are arbitrarily large $L \in \mathcal L (H)$.

Sets of lengths in $H$ can be studied in the associated
monoid of zero-sum sequences $\mathcal B (G)$. The latter  is a Krull monoid again and it is well-known that $\mathcal L (H) = \mathcal L \big( \mathcal B (G) \big)$ (as usual we write $\mathcal L (G)$ for $\mathcal L \big( \mathcal B (G) \big)$. Our results on $\mathcal L (G)$ also apply to the larger class of  transfer Krull monoids over $G$. This will be outlined in Section \ref{2}.

If the group $G$ is infinite, then, by a Theorem of Kainrath, every finite set $L \subset \N_{\ge 2}$ lies in $\mathcal L (H)$ (\cite{Ka99a}, \cite[Theorem 7.4.1]{Ge-HK06a}; for further rings and monoids with this property see \cite{Fr-Na-Ri19a} or \cite[Corollary 4.7]{Go19b}).
Now suppose that the group $G$ is finite. In this case sets of lengths  have a well-defined structure. Indeed,
by the Structure Theorem for Sets of Lengths (Proposition \ref{2.2}.1), the sets in $\mathcal L (G)$ are AAMPs (almost arithmetical multiprogressions) with difference in $\Delta^* (G)$ and some universal bound. By a realization result of the second author, this description is best possible (Proposition \ref{2.2}.2). By definition, the concept of an AAMP comprises APs (arithmetical progressions), AAPs (almost arithmetical progressions), and  AMPs (arithmetical multiprogressions); definitions are gathered in Definition \ref{2.1}. The  goal of this paper is to characterize  those groups  where all sets of lengths are not only AAMPs, but have one of these more special forms.  We formulate the main result of this paper.

\begin{theorem} \label{1.1}
Let $G$ be a finite abelian group.
\begin{enumerate}
\item The following statements are equivalent{\rm \,:}
      \begin{enumerate}
      \item[(a)] All sets of lengths in $\mathcal L (G)$ are {\rm AP}s with difference in $\Delta^* (G)$.

      \item[(b)]  All sets of lengths in $\mathcal L (G)$ are {\rm AP}s.

      \item[(c)] The system of sets of lengths  $\mathcal L (G)$ is additively closed, that is, $L_1 +L_2 \in \mathcal L (G) $ for all $L_1, L_2 \in   \mathcal L (G) $.

      \item[(d)] $G$ is cyclic of order $|G| \le 4$ or isomorphic to a subgroup of  $C_2^3$ or isomorphic to a subgroup of  $C_3^2$.
\end{enumerate}

\item The following statements are equivalent{\rm \,:}
      \begin{enumerate}
      \item[(a)] There is a constant $M \in \mathbb N$ such that all sets of lengths in $\mathcal L (G)$ are {\rm AAP}s with bound $M$.

      \item[(b)] $G$ is isomorphic to a subgroup of $C_3^3$ or isomorphic to a subgroup of $C_4^3$.
      \end{enumerate}

\item For all finite sets $\Delta$ with $\Delta \supset \Delta^* (G)$ the following statements are equivalent{\rm \,:}
      \begin{enumerate}
      \item[(a)] All sets of lengths in $\mathcal L (G)$ are {\rm AMP}s with difference in $\Delta$.

      \item[(b)] $G$ is cyclic of order $|G| \le 6$  or isomorphic to a subgroup of $C_2^3$ or isomorphic to a subgroup of $C_3^2$.
      \end{enumerate}
\end{enumerate}
\end{theorem}

A central topic in the study of sets of lengths is the Characterization Problem (for recent progress see \cite{B-G-G-P13a, Ge-Sc16a, Ge-Zh17b, Zh18a, Zh18b}) which reads as follows:

\begin{enumerate}
\item[]
{\it Let $G$ be a finite abelian group with $\mathsf D (G)\ge 4$, and let  $G'$ be an abelian group with $\mathcal L(G) = \mathcal L(G')$.
Does it follow that $G \cong G'$?}
\end{enumerate}

A finite abelian group $G$ has Davenport constant $\mathsf D (G) \le 3$ if and only if either $|G|\le 3$ or $G \cong C_2 \oplus C_2$. Since $\mathcal L (C_1)=\mathcal L (C_2)$ and $\mathcal L (C_3)=\mathcal L (C_2\oplus C_2)$ (Proposition \ref{3.1}), small groups require special attention in the study of the Characterization Problem.
As a consequence of Theorem \ref{1.1} we obtain an affirmative answer to the Characterization Problem for all involved small groups.

\begin{corollary} \label{1.2}
Let $G$ be a finite abelian group with Davenport constant $\mathsf D (G) \ge 4$ and suppose that $\mathcal L (G)$ satisfies one of the properties characterized in Theorem \ref{1.1}. If $G'$ is any abelian group such that $\mathcal L (G)=\mathcal L (G')$, then $G \cong G'$.
\end{corollary}

In Section \ref{2} we gather the required tools for studying sets of lengths (Propositions \ref{2.2}, \ref{2.3}, and \ref{2.4}). The proof of Theorem \ref{1.1} requires methods from additive combinatorics and is given in Section \ref{3}. Several properties occurring in Theorem \ref{1.1} can be characterized by further arithmetical invariants. We briefly outline this in Remark \ref{3.8} where we also discuss the property of being additively closed occurring in Theorem \ref{1.1}.1.(c).

\section{Background on sets of lengths} \label{2}

Let $\N$ denote the set of positive integers, $\P \subset \N$ the set of prime numbers and put $\N_0 = \N \cup \{0\}$. For real numbers $a,\,b \in \R$, we set $[a, b ] = \{ x \in \Z \mid a \le x \le b \}$. Let  $A, B \subset \mathbb Z$ be subsets of the integers.  We denote by $A+B = \{a+b \mid a \in A, b \in B\}$ their {\it sumset}, and by $\Delta (A)$ the {\it set of $($successive$)$ distances} of $A$ (that is, $d \in \Delta (A)$ if and only if $d = b-a$ with $a, b \in A$ distinct and $[a, b] \cap A = \{a, b\}$). For $k \in \N$, we denote by $k \cdot A = \{k a \mid a \in A \}$ the dilation of $A$ by $k$. If $A \subset \N$, then
\[
\rho (A) = \sup \Big\{ \frac{m}{n} \mid m, n \in A \Big\} = \frac{\sup A}{\min A} \in \Q_{\ge 1} \cup \{\infty\}
\]
is the {\it elasticity} of $A$, and we set $\rho ( \{0\})=1$.

\smallskip
\noindent
{\bf Monoids.} By a  {\it monoid},  we
mean a commutative cancellative semigroup with identity. Let $H$ be a monoid. Then $H^{\times}$ denotes the group of invertible elements of $H$, $\mathcal A (H)$ the set of atoms of $H$, and $H_{\rm red}=H/H^{\times}$ the associated reduced monoid. The factorization monoid $\mathsf Z (H)=\mathcal F (\mathcal A (H_{\rm red}))$ is the free abelian monoid with basis $\mathcal A (H_{\rm red})$ and $\pi \colon  \mathsf Z (H) \to H_{\rm red}$ denotes the canonical epimorphism. For an element $a \in H$,  $\mathsf Z (a) = \pi^{-1}(aH^{\times})$ is the {\it set of factorizations} of $a$ and $\mathsf L (a) = \{|z| = \sum_{u \in \mathcal A (H_{\rm red})} \mathsf v_u (z) \mid z = \prod_{u \in \mathcal A (H_{\rm red})} u^{\mathsf v_u (z)}  \in \mathsf Z (a) \} \subset \N_0$ is the {\it set of lengths} of $a$ (note that $\mathsf L (a) = \{0\}$ if and only if $a \in H^{\times}$). We denote by
\[
\begin{aligned}
\mathcal L (H) & = \left\{ \mathsf L (a) \mid a \in H \right\} \qquad \text{the {\it system of sets of lengths} of} \ H \,, \quad \text{and by} \\
\Delta (H) & = \bigcup_{L \in \mathcal L (H)} \Delta ( L ) \ \subset
\N \qquad \text{the {\it set of distances} of} \ H \,.
\end{aligned}
\]
For $k \in \N$, we set $\rho_k (H) = k$ if $H$ is a group, and
\[
\rho_k (H) = \sup \{ \sup L \mid L \in \mathcal L (H), k \in L \} \in \N \cup \{\infty\}, \quad \text{otherwise}  \,.
\]
Then
\[
\rho (H) = \sup \{ \rho (L) \mid L \in \mathcal L (H) \} = \lim_{k \to \infty} \frac{\rho_k (H)}{k} \in \R_{\ge 1} \cup \{\infty\}
\]
is the {\it elasticity} of $H$.

\smallskip
\noindent {\bf Zero-Sum Theory.} Let $G$ be an additive abelian group,   $G_0 \subset G$ a subset, and let $\mathcal F (G_0)$ be the free abelian monoid with basis $G_0$. In Combinatorial Number Theory, the elements of $\mathcal F(G_0)$ are called \ {\it sequences} over \
$G_0$.  For a sequence
\[
S = g_1 \cdot \ldots \cdot g_l = \prod_{g \in G_0} g^{\mathsf v_g
(S)} \in \mathcal F (G_0) \,,
\]
we set  $-S = (-g_1) \cdot \ldots \cdot (-g_l)$, and we call $|S|  = l = \sum_{g \in G} \mathsf v_g (S) \in \mathbb N_0 \
\text{the \ {\it length} \ of \ $S$,}$
\[
\begin{aligned}
\supp (S)&  = \{g \in G \mid \mathsf v_g (S) > 0 \} \subset G \ \text{the  {\it support} of
$S$} \,,  \mathsf v_g(S) \, \text{ the {\it multiplicity} of $g$ in $S,$}  \\
\sigma (S) & = \sum_{i = 1}^l g_i \ \text{the \ {\it sum} \ of \
$S$} \,, \  \text{and} \ \Sigma (S) = \Big\{ \sum_{i \in I} g_i
\mid \emptyset \ne I \subset [1,l] \Big\} \
\end{aligned}
\]
the  {\it set of subsequence sums} \ of \ $S$. The sequence $S$ is said to be
\begin{itemize}
\item {\it zero-sum free} \ if \ $0 \notin \Sigma (S)$,

\item a {\it zero-sum sequence} \ if \ $\sigma (S) = 0$,

\item a {\it minimal zero-sum sequence} \ if it is a nontrivial zero-sum
      sequence and every proper  subsequence is zero-sum free.
\end{itemize}
The monoid
\[
\mathcal B(G_0) = \{ S  \in \mathcal F(G_0) \mid \sigma (S) =0\} \subset \mathcal F (G_0)
\]
is called the \ {\it monoid of zero-sum sequences} \ over \ $G_0$. As usual  we set, for all $k \in \N$, $\rho_k (G_0) = \rho_k \big( \mathcal B (G_0) \big), \ \rho (G_0) = \rho \big( \mathcal B (G_0) \big), \ \mathcal L (G_0) = \mathcal L ( \mathcal B (G_0) )$, $\mathsf Z (G_0) = \mathsf Z \big( \mathcal B (G_0) \big)$,   and $\Delta (G_0) = \Delta \big( \mathcal B (G_0) \big)$.
The atoms (irreducible elements) of the monoid $\mathcal B (G_0)$ are precisely the
minimal zero-sum sequences over $G_0$, and they will be denoted by $\mathcal A (G_0)$. If $G_0$ is finite, then $\mathcal A (G_0)$ is finite. The  {\it Davenport constant} $\mathsf D (G_0)$ of $G_0$ is the maximal length of an atom whence
\[
\mathsf  D (G_0) = \sup \bigl\{ |U| \, \bigm| \; U \in \mathcal A (G_0) \bigr\} \in \N_0 \cup \{\infty\} \,.
\]
The {\it set of minimal distances} $\Delta^* (G) \subset \Delta (G)$ is  defined as
\[
\Delta^* (G) = \{\min \Delta (G_0) \mid G_0 \subset G \ \text{with} \ \Delta (G_0) \ne \emptyset \} \subset \Delta (G) \,.
\]
A tuple $(e_i)_{i \in I}$ is called a {\it basis} of $G$ if all elements are nonzero and $G = \oplus_{i \in I} \langle e_i \rangle$. For $p \in \P$, let $\mathsf r_p (G)$ denote the $p$-rank of
$G$, $\mathsf r (G) = \sup \{ \mathsf r_p (G) \mid p \in \P \}$ denote the
{\it rank} of $G$.

\smallskip
\noindent
{\bf Transfer Krull monoids.} A  monoid is a {\it Krull monoid} if it is completely integrally closed and satisfies the ascending chain condition on divisorial ideals.
An integral domain $R$ is a Krull domain if and only if its
multiplicative monoid $R \setminus \{0\}$ is a Krull monoid whence every integrally closed noetherian domain is Krull.
Rings of integers, holomorphy rings in algebraic function fields, and
regular congruence monoids in these domains are Krull monoids with finite
class group such that every class contains a prime divisor
(\cite[Section 2.11]{Ge-HK06a}). Monoid domains and monoids of modules that are Krull
 are discussed in \cite{Ch11a, Ba-Wi13a, Ba-Ge14b,Fa06a}. Let $H$ be a Krull monoid with class group $G$ such that every class contains a prime divisor. Then there is a transfer homomorphism $\theta \colon H \to \mathcal B (G)$ which implies that $\mathcal L (H)=\mathcal L (G)$ (\cite[Theorem 3.4.10]{Ge-HK06a}). A {\it transfer Krull monoid} $H$ over $G$ is a (not necessarily commutative) monoid allowing  a weak transfer homomorphism to  $\mathcal B (G)$ whence $\mathcal L (H)=\mathcal L (G)$. Thus Theorem \ref{1.1}  applies to transfer Krull monoids over finite abelian groups.
Recent deep work, mainly by Baeth and Smertnig, revealed that wide classes of non-commutative Dedekind domains are transfer Krull (\cite{Sm13a, Ba-Sm15,Sm19a}). We  refer to  the survey \cite{Ge16c} for a detailed discussion of these and further examples.

\smallskip
\noindent
{\bf Sets of Lengths.}
Let $A \in \mathcal B (G_0)$ and $d = \min \{ |U| \mid U \in \mathcal A (G_0) \}$. If $A = BC$ with $B, C \in \mathcal B (G_0)$, then
\[
\mathsf L (B) + \mathsf L (C) \subset \mathsf L (A) \,.
\]
If $A = U_1 \cdot \ldots \cdot U_k = V_1 \cdot \ldots \cdot V_l$ with $U_1, \ldots, U_k, V_1, \ldots, V_l \in \mathcal A (G_0)$ and $k < l$, then
\[
l d \le \sum_{\nu=1}^l |V_{\nu}| = |A| = \sum_{\nu=1}^k |U_{\nu}| \le k \mathsf D (G_0) \ \text{whence} \
\frac{|A|}{\mathsf D (G_0)} \le \min \mathsf L (A) \le \max \mathsf L (A) \le \frac{|A|}{d} \,.
\]
For sequences over cyclic groups  the $g$-norm plays a similar role as the length does for sequences over arbitrary groups.
Let $g \in G$ with \ $\ord(g) = n \ge 2$. For a
sequence $S = (n_1g) \cdot \ldots \cdot (n_lg) \in \mathcal F
(\langle g \rangle)$, \ where $l \in \N_0$ and $n_1, \ldots, n_l
\in [1, n]$, \ we define
\[
\| S \|_g = \frac{n_1+ \ldots + n_l}n  \,.
\]
Note that $\sigma (S) = 0$ implies that   $n_1 + \ldots + n_l
\equiv 0 \mod n$ whence $\| S \|_g \in \N_0$.  Thus, $\| \cdot
\|_g \colon \mathcal B(\langle g \rangle) \to \N_0$ is a
homomorphism, and $\| S \|_g =0$ if and only if $S=1$. If $S \in \mathcal A (G_0)$, then $\|S\|_g \in [1, n-1]$, and if $\|S\|_g = 1$, then $S \in \mathcal A (G_0)$.
Arguing as above we obtain that
\[
\frac{\|A\|_g}{n-1} \le \min \mathsf L (A) \le \max \mathsf L (A) \le \|A\|_g \,.
\]

Now we recall the concept of almost arithmetical multiprogressions (AAMPs) as given in \cite[Chapter 4]{Ge-HK06a}. Then we gather  results on sets of lengths and on invariants controlling their structure such as the set of distances and the elasticities (Propositions \ref{2.2}, \ref{2.3}, and {2.4}). These results form the basis for the proof of Theorem \ref{1.1} given in the next section.

\begin{definition} \label{2.1}

Let $d \in \N$,  $l,\, M \in \N_0$, and  $\{0,d\} \subset
\mathcal D \subset [0,d]$. A subset $L \subset \Z$ is called an

\begin{itemize}
\item  {\it arithmetical multiprogression} \ ({\rm AMP} \ for short) \ with \ {\it   difference} \ $d$, \ {\it period} \ $\mathcal D$ \ and \ {\it length}
      \ $l$, \ if $L$ is an interval of \ $\min L + \mathcal D + d \Z$ \ (this means that $L$ is finite nonempty and $L = (\min L + \mathcal
D + d \mathbb Z) \cap [\min L, \max L]$),
      \ and \ $l$ \ is maximal such that \ $\min L + l d \in L$.

\item {\it almost arithmetical multiprogression} \ ({\rm AAMP} \ for
      short) \ with \ {\it difference} \ $d$, \ {\it period} \ $\mathcal D$,
      \ {\it length} \ $l$ and \ {\it bound} \ $M$, \ if
      \[
      L = y + (L' \cup L^* \cup L'') \, \subset \, y + \mathcal D + d \Z
      \]
      where \ $L^*$ \ is an \ {\rm AMP} \ with difference $d$ (whence $L^* \ne \emptyset$),
      period $\mathcal D$
      and length $l$ such that \ $\min L^* = 0$, \ $L' \subset [-M, -1]$, \ $L'' \subset \max L^* + [1,M]$ \ and \
      $y \in \Z$.

\item {\it almost arithmetical progression} \ ({\rm AAP} \ for short) \ with
      \ {\it difference} $d$, \ {\it bound} $M$ \ and \ {\it length} \
      $l$, \ if it is an \,{\rm AAMP}\, with difference $d$, period \ $\{0,d\}$, bound \ $M$
      \ and length \ $l$.
\end{itemize}
\end{definition}

\begin{proposition}[Structural results on $\mathcal L (G)$] \label{2.2}~

Let $G$ be a finite abelian group with $|G|\ge 3$.
\begin{enumerate}
\item There exists some $M \in \N_0$  such that every set of lengths  $L \in \mathcal L(G)$ is an {\rm AAMP} with some difference $d \in \Delta^* (G)$ and bound $M$.

\item For every $M \in \mathbb N_0$ and every finite nonempty set $\Delta^* \subset \mathbb N$, there is a finite abelian group $G^*$ such that the following holds{\rm \,:} for every {\rm AAMP} $L$ with difference $d\in \Delta^*$ and bound $M$ there is some $y_{L} \in \mathbb N$ such that
      \[
      y+L \in \mathcal L (G^*) \quad \text{ for all } \quad y \ge y_{L} \,.
      \]

\item Let $G_0 \subset G$ be a subset. Then there exist a bound $M \in \N_0$ and some $A^* \in \mathcal B (G_0)$ such that for all $A \in A^* \mathcal B (G_0)$ the set of lengths $\mathsf L (A)$ is an {\rm AAP} with difference $\min \Delta (G_0)$ and bound $M$.

\item If $A \in \mathcal B (G)$ such that $\supp (A) \cup \{0\}$ is a subgroup of $G$, then $\mathsf L (A)$ is an {\rm AP} with difference $1$.
\end{enumerate}
\end{proposition}

\begin{proof}
The first statement gives the Structure Theorem for Sets of Lengths (\cite[Theorem 4.4.11]{Ge-HK06a}), which is sharp by the second statement proved in \cite{Sc09a}. The third and the fourth statements show that sets of lengths are extremely smooth provided that the associated zero-sum sequence contains all elements of its support sufficiently often (\cite[Theorems 4.3.6 and 7.6.8]{Ge-HK06a}).
\end{proof}

\begin{proposition}[Structural results on $\Delta (G)$ and on $\Delta^* (G)$] \label{2.3}~

Let $G = C_{n_1} \oplus \ldots \oplus C_{n_r}$ where $r, n_1, \ldots, n_r \in \N$ with $r=\mathsf r (G)$, $1 < n_1 \t \ldots \t n_r$, and $|G|\ge 3$.
\begin{enumerate}
\item $\Delta (G)$ is an interval with
      \[
      \bigl[1,\, \max \{\exp (G) -2, \, k-1 \} \bigr] \subset \Delta (G) \subset \bigl[1, \mathsf D (G) - 2\bigr] \ \text{ where } \  k =\sum_{i=1}^{\mathsf r (G)} \Bigl\lfloor \frac{n_i}{2} \Bigr\rfloor  \,.
      \]

\item $1 \in \Delta^* (G) \subset \Delta (G)$,  $[1, \mathsf r (G)-1] \subset \Delta^* (G)$, and $\max \Delta^* (G) = \max \{\exp (G)-2, \mathsf r (G)-1 \}$.

\item If $G$ is cyclic of order $|G| = n \ge 4$, then $\max \big( \Delta^* (G) \setminus \{n-2\}\big) = \lfloor \frac{n}{2} \rfloor - 1$.
\end{enumerate}
\end{proposition}

\begin{proof}
The statement on $\max \Delta^* (G)$ follows from \cite{Ge-Zh16a}.
For all remaining statements see  \cite[Section 6.8]{Ge-HK06a}. A more detailed analysis of $\Delta^* (G)$ in case of cyclic groups can be found in \cite{Pl-Sc18a}.
\end{proof}

\begin{proposition}[Results on $\rho_k (G)$ and on $\rho (G)$] \label{2.4}~

Let $G$ be a finite abelian group with $|G| \ge 3$, and let $k \in \N$.
\begin{enumerate}
\item $\rho (G) = \mathsf D (G)/2$ and $\rho_{2k} (G) = k \mathsf D (G)$.

\item $1 + k \mathsf D (G) \le \rho_{2k+1} (G) \le k \mathsf D (G) + \mathsf D (G)/2$. If $G$ is cyclic, then equality holds on the left side.
\end{enumerate}
\end{proposition}

\begin{proof}
See \cite[Chapter 6.3]{Ge-HK06a}, \cite[Theorem 5.3.1]{Ge09a}, and \cite{Ge-Gr-Yu15}.
\end{proof}

We need the  following technical lemma which follows from \cite[Lemma 6.1]{Fr-Sc10}.

\begin{lemma} \label{wichtig-0}
Let $G$ be a finite abelian group, $A \in \mathcal B (G)$, $x = U_1 \cdot \ldots \cdot U_r \in \mathsf Z (G)$, with $U_1, \ldots, U_r \in \mathcal A (G)$ and $|U_1|= \ldots = |U_r|=2$, such that $U_1 \cdot \ldots \cdot U_r \mid A $ in $\mathcal B (G)$. Then there exists a factorization $z \in \mathsf Z (A)$ with $|z|= \max \mathsf L (A)$ and $x \mid z$ in $\mathsf Z (G)$.
\end{lemma}

\section{A characterization of extremal cases} \label{3}

The goal of this section is to prove Theorem \ref{1.1}. First we  recall some cases where the systems of sets of lengths are completely determined.

\begin{proposition} \label{3.1}~
\begin{enumerate}
\item $\mathcal L (C_1) = \mathcal L (C_2) = \big\{ \{m\} \mid m \in \N_0 \big\}$.

\item $\mathcal L (C_3) = \mathcal L (C_2 \oplus C_2) = \bigl\{ y
      + 2k + [0, k] \, \bigm| \, y,\, k \in \N_0 \bigr\}$.

\item $\mathcal L (C_4) = \bigl\{ y + k+1 + [0,k] \, \bigm|\, y,
      \,k \in \N_0 \bigr\} \,\cup\,  \bigl\{ y + 2k + 2 \cdot [0,k] \, \bigm|
      \, y,\, k \in \N_0 \bigr\} $.

\item $\mathcal L (C_2^3)  =  \bigl\{ y + (k+1) + [0,k] \,
      \bigm|\, y
      \in \N_0, \ k \in [0,2] \bigr\}$ \newline
      $\quad \text{\, } \ \qquad$ \quad $\cup \ \bigl\{ y + k + [0,k] \, \bigm|\, y \in \N_0, \ k \ge 3 \bigr\}
      \cup \bigl\{ y + 2k
      + 2 \cdot [0,k] \, \bigm|\, y ,\, k \in \N_0 \bigr\}$.

\item  $\mathcal L (C_3^2) = \{ [2k, l] \mid k \in \mathbb N_0, l \in [2k, 5k]\}$ \newline
 $\quad \text{\, } \ \qquad$ \quad $\cup \ \{ [2k+1, l] \mid k \in \N, l \in [2k+1, 5k+2] \} \cup \{ \{ 1\}  \}$.
\end{enumerate}
\end{proposition}

\begin{proof}
1. This is  well-known. A proof of 2.,3., and 4. can be found in \cite[Theorem 7.3.2]{Ge-HK06a}.
For 5. we refer to \cite[Proposition 3.12]{Ge-Sc16b}.
\end{proof}

Let $G$ be a finite abelian group, $g \in G$ with $\ord (g) = n \ge 2$, $r \in [0, n-1]$, $k, \ell \in \N_0$, and $A = g^{kn+r}(-g)^{\ell n + r}$. We often use that
\begin{equation} \label{basic}
\mathsf Z (A) = \{ U^{r+n \nu} V^{k - \nu} (-V)^{\ell - \nu} \mid \nu \in [0, \min \{k, \ell \}] \} \,,
\end{equation}
where $U = g(-g)$, $V = g^n$, and that
\[
\begin{aligned}
\mathsf L (A) & = \{r+k+\ell + \nu (n-2) \mid \nu \in [0, \min \{k, \ell \}] \} \\
 & = r+k+\ell + (n-2) \cdot [0, \min \{k, \ell \}] \,.
\end{aligned}
\]

\begin{proposition} \label{3.2}
Let $G$ be a cyclic group of order $|G| = n \ge 7$,   $g \in G$ with $\ord (g) = n$, $k \in \N$, and
\[
A_k = \begin{cases} g^{nk} (-g)^{nk}(2g)^{n} & \quad \text{if} \ n \ \text{is even}, \\
                     g^{nk} (-g)^{nk} \big( (2g)^{(n-1)/2} g \big)^2 & \quad \text{if} \ n \ \text{is odd}.
      \end{cases}
\]
Then there is a bound $M \in \N$ such that, for all  $k \ge n-1$, the sets $\mathsf L (A_k)$ are {\rm AAP}s with difference $1$ and bound $M$, but they are not {\rm AP}s with difference $1$.
\end{proposition}

\begin{proof}
We set $G_0 = \{g, -g, 2g\}$, $U_1 = (-g)g$, $U_2= (-g)^2 (2g)$ and, if $n$ is odd, then $V_1 = (2g)^{(n+1)/2}(-g)$. Furthermore, for $j \in [0, n/2]$, we define $W_j = (2g)^j g^{n-2j}$. Then, together with $-W_0 = (-g)^n$, these are all minimal zero-sum sequences which divide $A_k$ for $k \in \N$. Note that
\[
\|-W_0\|_g = n-1, \ \|U_2\|_g = \|V_1\|_g = 2,  \  \text{and} \  \|U_1\|_g = \|W_j\|_g =1 \quad  \text{for all} \ j \in [0, n/2] \,.
\]

It is sufficient to prove the following two assertions.
\begin{enumerate}
\item[{\bf A1.}\,]  There is an $M \in \N_0$ such that $\mathsf L (A_k)$ is an AAP with difference $1$ and bound $M$ for all $k \ge n-1$.

\item[{\bf A2.}\,]  For each $k \in \N$,  $\mathsf L (A_k)$ is not an AP with difference $1$.
\end{enumerate}

\smallskip

{\it Proof of \,{\bf A1}}.\, By Proposition \ref{2.2}.1 there is a bound $M' \in \mathbb N_0$ such that, for each $k \in \N$, $\mathsf L (A_k)$ is an AAMP with difference $d_k \in \Delta^* (G) \subset [1, n-2]$ and bound $M'$. Suppose that $k \ge n-1$. Then $(W_0U_2)^{n-1}$ divides $A_k$. Since $W_0 U_2 = W_1U_1^2$, it follows that
\[
(W_0 U_2)^{n-1} = (W_0 U_2)^{n - 1- \nu} (W_1U_1^2)^{\nu} \quad \text{for all} \quad \nu \in [0,n-1]
\]
and hence $\mathsf L \big( (W_0 U_2)^{n-1} \big) \supset [2n-2, 3n-3]$. Thus $\mathsf L (A_k)$ contains an AP with difference $1$ and length $n-1$. Therefore there is a bound $M \in \N_0$ such that $\mathsf L (A_k)$ is an AAP with difference $1$ and bound $M$ for all $k \ge n-1$.

{\it Proof of \,{\bf A2}}.\, Let $k \in \N$. Observe that
\[
A_k = \begin{cases} W_0^{k} (-W_0)^{k} W_{n/2}^{2} & \quad   \ \text{if} \ n \ \text{is even}, \\
                     W_0^{k} (-W_0)^{k} \big( W_{(n-1)/2})^2 & \quad  \ \text{if} \ n \ \text{is odd},
      \end{cases}
\]
and it can be seen that $\min \mathsf L (A_k) = 2k+2$. We assert that $2k+3 \notin \mathsf L (A_k)$. If $n$ is even, then
\[
W_0 W_{n/2} = W_j W_{n/2-j} \quad \text{for each} \quad j \in [0, n/2] \,,
\]
and similarly, for odd $n$ we have
\[
W_0 W_{(n-1)/2} = W_j W_{(n-1)/2 - j} \quad \text{for each} \quad j \in [0, (n-1)/2] \,.
\]
In both cases, all factorizations of $A_k$ of length $2k+2$ contain only atoms with $g$-norm  $1$ and with $g$-norm  $n-1$. Let $z'$ be any factorization of $A_k$ containing only atoms with $g$-norm  $1$ and with $g$-norm $n-1$. Then $|z'| - |z|$ is a multiple of $n-2$ whence if $|z'| > |z|$, then $|z'|-|z| \ge n-2 > 1$.

Next we consider a factorization $z'$ of $A_k$ containing at least one atom with $g$-norm  $2$, say $z'$ has $r$ atoms with $g$-norm  $n-1$, $s \ge 1$ atoms with $g$-norm  $2$, and $t$ atoms with $g$-norm  $1$. Then $k > r$,
\[
\|A_k\|_g = k(n-1) + (k+2) = r(n-1)+ 2s + t \,,
\]
and we study
\[
\begin{aligned}
|z'|-|z| & = r+s+t - (2k+2) \\
 & = r+s+k(n-1)+(k+2)-r(n-1)-2s-(2k+2) \\
 & = (k-r)(n-2)-s \,.
\end{aligned}
\]
Note that $s \le \mathsf v_{2g} (A_k) \le n$. Thus, if $k-r \ge 2$, then
\[
(k-r)(n-2)-s \ge 2n-4-s \ge n-4 > 1 \,.
\]
Suppose that $k-r=1$. Then we cancel $(-W_0)^{k-1}$, and consider a relation where $-W_0$ occurs precisely once. Suppose that all $s$ atoms of $g$-norm $2$ are equal to $U_2$. Since $\mathsf v_{-g} (U_2) = 2$, it follows that $s \le \mathsf v_{-g}(-W_0)/2 = n/2$ whence
\[
(k-r)(n-2) - s \ge n-2 - n/2 = n/2 - 2 > 1 \,.
\]
Suppose that $V_1$ occurs among the $s$ atoms with $g$-norm $2$. Then $n$ is odd, $V_1$ occurs precisely once,  and
\[
s-1 \le \mathsf v_{2g} (A_k) - \frac{n+1}{2} = (n-1) - \frac{n+1}{2} = \frac{n-3}{2} \,,
\]
whence
\[
(k-r)(n-2) - s \ge (n-2) - \frac{n-1}{2} = \frac{n+1}{2} - 2 > 1 \,. \qedhere
\]
\end{proof}

In order to handle the cyclic group with six elements (done in Proposition \ref{3.3}), we need five lemmas.

\begin{lemma} \label{wichtig-1}
Let $G$ be a cyclic group of order $|G|=6$, $g \in G$ with $\ord (g)=6$, and $A \in \mathcal B (G)$ with $\supp (A) = \{g, 2g, 3g, - g\}$. Then $\mathsf L (A)$ is an {\rm AP} with difference $1$.
\end{lemma}

\begin{proof}
We set $G_0 = \{g, 2g, 3g, - g\}$ and list all atoms of $\mathcal B (G_0)$:
\begin{itemize}
\item[] $\| \cdot  \|_g = 1:$ $g^6$, $g^4 (2g)$, $ g^3(3g)$, $g^2(2g)^2$, $g(2g)(3g)$, $g(-g)$, $(2g)^3$, $(3g)^2$

\item[] $\| \cdot  \|_g = 2:$ $U_1 = (2g)^2(3g)(-g)$, $U_2 = (2g)(-g)^2$,

\item[] $\| \cdot  \|_g = 3:$ $U_3 = (3g)(-g)^3$,

\item[] $\| \cdot  \|_g = 5:$ $U_4=(-g)^6$.
\end{itemize}

We choose a factorization $z \in \mathsf Z (A)$ and we show that either $|z|= \max \mathsf L (A)$ or that $|z|+1 \in \mathsf L (A)$.
In each of the following cases it easily follows that $|z|+1 \in \mathsf L (A)$.
\begin{itemize}
\item $U_1^2 \mid z$ in $\mathsf Z (G)$, because $3 \in \mathsf L (U_1^2)$.

\item $U_1U_3 \mid z$ in $\mathsf Z (G)$, because $3 \in \mathsf L (U_1 U_3)$.

\item $U_1U_4 \mid z$ in $\mathsf Z (G)$, because $3 \in \mathsf L (U_1 U_4)$.
\end{itemize}
From now on we suppose that none of the above cases occurs. If $z$ is divisible only by atoms with $g$-norm one and by at most one atom having $g$-norm two, then we are done. Thus it is sufficient to handle the following four cases.

\noindent
CASE 1: $U_1 \mid z$.

If $z$ is divisible by an atom $W$ with $\mathsf v_g (W) \ge 1$ and $W \ne (-g)g$, then $3 \in \mathsf L (U_1W)$ and we are done. Suppose this does not hold. Thus
\[
z = U_1 U_2^{k_2} \Big( (2g)^3 \Big)^{k_3} \Big( (3g)^2 \Big)^{k_4} \Big( (-g)g \Big)^{k_5} \,,
\]
with $k_2 \in \N$ and $k_3,k_4,k_5 \in \N_0$. Then $|z|=1+k_2 + \ldots + k_5$ and we assert that $|z|=\max \mathsf L (A)$. By Lemma \ref{wichtig-0} it is sufficient to show that
\[
1 + k_2+k_3 = \max \mathsf L \Big(  U_1 U_2^{k_2} \Big( (2g)^3 \Big)^{k_3}   \Big) \,.
\]
Each atom $X \in \mathcal B (G_0)$ with $X \mid U_1 U_2^{k_2} \Big( (2g)^3 \Big)^{k_3}$ and $(2g)\mid X$ has length $|X| \ge 3$ whence the claim follows.

\noindent
CASE 2: $U_3 \mid z$ and $U_1 \nmid z$.

If $z$ is divisible by an atom $W$ with $\mathsf v_{2g} (W) \ge 1$. If $W \in \{ g^4(2g), g^2(2g)^2,  (2g)^3,$ $g(2g)(3g)\}$, then $3 \in \mathsf L (U_3W)$ and we are done. If this is not the case, then all atoms $W$ are equal to $U_2$. Thus $U_2U_3$ divides $z$. Thus $z$ is divisible by an atom $W$ with $\mathsf v_g (W) \ge 1$. If $W \in \{g^6, g^3(3g) \}$, then $3 \in \mathsf L (U_2W)$ and we are done. Thus we may assume that any atom $W$ with $\mathsf v_g (W) \ge 1$ is equal to $V=(-g)g$. This implies that $\mathsf v_g (A) = \mathsf v_V (z)$. We set
\[
z = U_2^{k_2}U_3^{k_3}U_4^{k_4} V^{k_4} \Big( (3g)^2 \Big)^{k_5}
\]
and assert that $|z|=\max \mathsf L (A)$. By Lemma \ref{wichtig-0} it is sufficient to show that
\[
k_2+k_3+k_4 = \max \mathsf L \Big( U_2^{k_2}U_3^{k_3}U_4^{k_4}  \Big) \,,
\]
and this holds true.

\noindent
CASE 3: $U_2 \mid z$ and $U_1 \nmid z$ and $U_3 \nmid z$.

If $z$ is divisible by an atom $W$ with $\mathsf v_{2g} (W) \ge 2$, then $3 \in \mathsf L (U_2W)$ and we are done. If $z$ is divisible by $W = g(2g)(3g)$, then $U_2W= U_1 \Big( g(-g) \Big)$. Thus there is a factorization $z' \in \mathsf Z (A)$ with $|z|=|z'|$ and $U_1 \mid z'$ and the claim follows from CASE 1.

Suppose that this is not the case. Then the each atom dividing $z$ and containing $g$ is $V=(-g)g$ whence $\mathsf v_g (A) = \mathsf v_V (z)$. If $(2g)^3$ and $U_4$ divide $z$, then $|z|+1 \in \mathsf L (A)$ because $3 \in \mathsf L ( (3g)^2 U_4)$; otherwise $|z|= \max \mathsf L (A)$.

\noindent
CASE 3: $U_4 \mid z$ and $U_1 \nmid z$ and $U_2 \nmid z$ and $U_3 \nmid z$.

If $W= (2g)^3$ divides $z$, then $3 \in \mathsf L (U_4W)$ and we are done. Suppose this is not the case. If $W = g(2g)(3g)$ divides $z$, then $3 \in \mathsf L (U_4W)$ and we are done. Suppose this is not the case.  Then there are atoms $W_1, W_2$ such that $2g \mid W_1$ and $3g \mid W_2$ such that $W_1W_2$ divides $z$. Then $W_1 \in \{ g^4(2g), g^2(2g)^2\}$ and $W_2 \in \{(3g)^2, g^3(3g)\}$. In each case we obtain that $W_1W_2 = W W_4$ for some atom $W_4 \in \mathcal A (G)$. Thus we obtain a factorization $z' \in \mathsf Z (A)$ with $|z|=|z'|$ with $WU_4 \mid z'$ and we are done.
\end{proof}

\begin{lemma} \label{wichtig-2A}
Let $G$ be a cyclic group of order $|G|=6$, $g \in G$ with $\ord (g)=6$, and $A \in \mathcal B (G)$ with $\supp (A) = \{g, 3g, 4g\}$. Then $\mathsf L (A)$ is an {\rm AP} with difference $1$.
\end{lemma}

\begin{proof}
We set $G_0 = \{g,  3g, 4g\}$ and list all atoms of $\mathcal B (G_0)$:
\begin{itemize}
\item[] $\| \cdot  \|_g = 1:$ $g^6$,  $V_3= g^3(3g)$,  $V_5=g^2(4g)$,   $V_4=(3g)^2$

\item[] $\| \cdot  \|_g = 2:$ $U_1 = (4g)^2(3g)g$, $U_2 = (4g)^3$,
\end{itemize}

We choose a factorization $z \in \mathsf Z (A)$ and we show that either $|z|= \max \mathsf L (A)$ or that $|z|+1 \in \mathsf L (A)$.
If $z$ is divisible by at most one atom having $g$-norm two, then the claim holds. Thus from now on we suppose that $z$ is divisible by at least two atoms having $g$-norm two.
Since $U_1^2 = (3g)^2 \Big( (4g)g^2 \Big) \Big( (4g)^3\Big)$, we may suppose that $U_1^2 \nmid z$. Thus $U_1U_2$ divides $z$ or $U_2^2$ divides $z$.

Since $V_3^2 = (g^6) \Big( (3g)^2 \Big)$, we may suppose that $\mathsf v_{V_3} (z) \le 1$.
If $z$ is divisible by $ g^6$, then $3 \in \mathsf L (U_2g^6)$ shows that $|z|+1 \in \mathsf L (A)$. Suppose that $g^6 \nmid z$. It follows that
\[
z =  U_1^{\ell_1}U_2^{\ell_2}  V_3^{\ell_3} V_4^{\ell_4}V_5^{\ell_5} \,.
\]
where $\ell_1 \le 1$, $\ell_3 \le 1$. Since $U_1V_3=V_4V_5^2$ we may suppose that $\ell _1+ \ell_3 \le 1$.
We assert that $|z|= \max \mathsf L (A)$, and by Lemma \ref{wichtig-0} it suffices to show that
\[
\ell_1+\ell_2+\ell_3+\ell_5 = \max \mathsf L (A')\,, \text{where} \ A' = U_1^{\ell_1}U_2^{\ell_2}V_3^{\ell_3}V_5^{\ell_5} \,.
\]
Since $A'$ is not divisible by an atom of length $2$, it follows that $\max \mathsf L (A') \le |A'|/3$. Since $|U_2|=|V_5|=3$ and $\ell_1 + \ell_3 \le 1$, we infer that $\ell_1+\ell_2+\ell_3+\ell_5 = \lfloor |A'|/3 \rfloor$.
\end{proof}

\begin{lemma} \label{wichtig-2}
Let $G$ be a cyclic group of order $|G|=6$, $g \in G$ with $\ord (g)=6$, and $A \in \mathcal B (G)\!$ with $\supp (A) = \!\{g, 2g, 3g, 4g\}.\!$ Then $\!\mathsf L (A)\!$ is an {\rm AP} with difference $1$.
\end{lemma}

\begin{proof}
We set $G_0 = \{g, 2g, 3g, 4g\}$ and list all atoms of $\mathcal B (G_0)$:
\begin{itemize}
\item[] $\| \cdot  \|_g = 1:$ $g^6$, $g^4 (2g)$, $V_3= g^3(3g)$, $g^2(2g)^2$, $V_5=g^2(4g)$, $V_1=g(2g)(3g)$, \\
    \phantom{$\| \cdot  \|_g = 1:$} $V_2=(2g)(4g)$,   $(2g)^3$, $V_4=(3g)^2$

\item[] $\| \cdot  \|_g = 2:$ $U_1 = (4g)^2(3g)g$, $U_2 = (4g)^3$,
\end{itemize}

We choose a factorization $z \in \mathsf Z (A)$ and we show that either $|z|= \max \mathsf L (A)$ or that $|z|+1 \in \mathsf L (A)$.
 If $z$ is divisible by at most one atom having $g$-norm two, then the claim holds. Thus from now on we suppose that $z$ is divisible by at least two atoms having $g$-norm two.
Since $U_1^2 = (3g)^2 \Big( (4g)g^2 \Big) \Big( (4g)^3\Big)$, we may suppose that $U_1^2 \nmid z$. Thus $U_1U_2$ divides $z$ or $U_2^2$ divides $z$.

Since $V_3^2 = (g^6) \Big( (3g)^2 \Big)$, we may suppose that $\mathsf v_{V_3} (z) \le 1$. Since $V_1^2 = \Big( g^2(2g)^2 \Big) (3g)^2$, we may suppose that $\mathsf v_{V_1} (z) \le 1$. Since $V_1V_3 = \Big( g^4 (2g) \Big) (3g)^2$, we may suppose that $\mathsf v_{V_3} (z) + \mathsf v_{V_1}(z) \le 1$. Since $V_1V_5 =  V_2V_3$, we may suppose that $\mathsf v_{V_5} (z) + \mathsf v_{V_1}(z) \le 1$.
If $z$ is divisible by some $W \in \{ g^6, g^4(2g), g^2(2g)^2, (2g)^3\}$, then $3 \in \mathsf L (U_2W)$ shows that $|z|+1 \in \mathsf L (A)$. Suppose that none of these four atoms divides $z$.

Since $\mathsf v_{2g}(A) \ge 1$, it follows that $V_1 = g(2g)(3g)$ or $V_2 = (2g)(4g)$ divides $z$. If $U_1V_1 \mid z$, then $3 \in \mathsf L (U_1V_1)$ shows that $|z|+1 \in \mathsf L (A)$. Since $U_1 V_2 = U_2 V_1$, we may suppose that
 $U_1 \nmid z$ whence $z$ has the form
\[
z = V_3^{k_3}V_5^{k_5}V_1^{k_1}V_2^{k_2}V_4^{k_4} U_2^{\ell} \,.
\]
We assert that $|z|= \max \mathsf L (A)$. By Lemma \ref{wichtig-0} it is sufficient to show that
\[
\max \mathsf L \Big( V_3^{k_3}V_5^{k_5}V_1^{k_1} U_2^{\ell} \Big) = k_3+k_5+k_1+\ell \,.
\]
If $k_1=1$, then $k_3=k_5=0$ and $\mathsf L \Big( V_1^{1} U_2^{\ell} \Big) = \{\ell+1\}$. Suppose that $k_1=0$, then $V_3^{k_3}V_5^{k_5}U^{\ell}$ is not divisible by an atom of length two whence (recall that $k_3 \le 1$)
\[
\max \mathsf L \Big( V_3^{k_3}V_5^{k_5}U^{\ell} \Big) \le \left\lfloor \frac{|V_3^{k_3}V_5^{k_5}U^{\ell}|}{3} \right\rfloor = k_3+k_5+\ell \,. \qedhere
\]
\end{proof}

\begin{lemma} \label{wichtig-3}
Let $G$ be a cyclic group of order $|G|=6$, $g \in G$ with $\ord (g)=6$, and $A \in \mathcal B (G)$ with $\supp (A) = \{g, 2g,  -g\}$. If $\mathsf v_{2g} (A) \ge 3$, then $\mathsf L (A)$ is an {\rm AP} with difference $1$. If $\mathsf v_{2g} (A) = 2$, then $\mathsf L (A)$ is an {\rm AMP} with period $\{0,1,2,4\}, \{0,1,3,4\}$, or $\{0,2,3,4\}$,  and if $\mathsf v_{2g} (A) = 1$, then it is an {\rm AMP} with period $\{0,1,4\}$ or $\{0,3,4\}$.
\end{lemma}

\begin{proof}
We set $G_0 = \{g, 2g, 3g, 4g\}$ and list all atoms of $\mathcal B (G_0)$:
\begin{itemize}
\item[] $\| \cdot  \|_g = 1:$ $U_1 = g^6, U_2 = g^4(2g), U_3 = g^2 (2g)^2, U_4 = (-g)g, U_5 = (2g)^3$,

\item[] $\| \cdot  \|_g = 2:$ $V_1 = (2g)(-g)^2$,

\item[] $\| \cdot  \|_g = 5:$ $V_2 = (-g)^6$.
\end{itemize}
CASE 1: $\mathsf v_{2g} (A)=1$.

Then each factorization $z \in \mathsf Z (A)$ is divisible by $U_2$ or $V_1$.
If $\mathsf v_{g} (A) < 4$ or $\mathsf v_{-g} (A) < 2$, it is easy to see that $\mathsf L (A)$ is a singleton. Thus, we assume that neither is the case and write $A = (2g)g^{v+4}(-g)^{w+2}$. Since $A$ is a zero-sum sequence, it follows that  $v\equiv w +2 \pmod{6}$.

We determine the set of lengths of $A$. Let $z$ be a factorization of $A$.
If $z$ contains $U_2$, then $z$ is of the form $U_2 z'$ where $z'$ is a factorization of $g^v (-g)^{w+2}$.
If $z$ contains $V_1$, then $z$ is of the form $V_1 z''$ where $z''$ is a factorization of $g^{v+4}(-g)^{w}$.
Conversely, for each factorization $z'$  of $g^v (-g)^{w+2}$ we have that $U_2 z'$ is a factorization of $A$, and
for each factorization $z''$  of $g^v (-g)^{w+2}$ we have that $V_1 z''$ is a factorization of $A$.
Consequently $\mathsf L (A) = 1 + (\mathsf L (  g^v (-g)^{w+2} )  \cup  \mathsf L (  g^{v+4} (-g)^{w}))$.

We determine $\mathsf L (  g^v (-g)^{w+2} )  \cup  \mathsf L (  g^{v+4} (-g)^{w})$.
If $v$ is odd, then each factorization contains of $g^v (-g)^{w+2}$ and $g^{v+4} (-g)^{w}$ contains $(-g)g$ and
$\mathsf L (  g^v (-g)^{w+2} )  \cup  \mathsf L (  g^{v+4} (-g)^{w}) = 1 + (\mathsf L (  g^{v-1} (-g)^{w-1+2} )  \cup  \mathsf L (  g^{v-1+4} (-g)^{w-1}))$.
We thus focus on the case that $v$ is even.

If $v \equiv 0 \pmod{6}$, then
\[\mathsf L ( g^v (-g)^{w+2}) = \frac{v+w+2}{6} + 4 \cdot [0, \frac{\min\{v,w+2\}}{6}]\]
and
\[\mathsf L ( g^{v+4} (-g)^{w}) = 4 + \frac{v+w-4}{6} + 4 \cdot [0, \frac{\min\{v,w-4\}}{6}] = 3 + \frac{v+w+2}{6} + 4 \cdot [0, \frac{\min\{v,w-4\}}{6}].\]
Thus, the union is an AMP with difference $4$ and period $\{0,3,4\}$.

If $v \equiv 2 \pmod{6}$, then
\[\mathsf L ( g^v (-g)^{w+2}) = 2 + \frac{v-2 +  w}{6} + 4 \cdot [0, \frac{\min\{v-2,w\}}{6}]= 1 + \frac{v+4 +  w}{6} + 4 \cdot [0, \frac{\min\{v-2,w\}}{6}]\]
and
\[\mathsf L ( g^{v+4} (-g)^{w}) =  \frac{v+4 +w}{6} + 4 \cdot [0, \frac{\min\{v+4,w\}}{6}].\]
Thus, the union is an AMP with difference $4$ and period $\{0,1,4\}$.

If $v \equiv 4 \pmod{6}$, then
\[
\begin{aligned}
\mathsf L ( g^v (-g)^{w+2}) & = 4+ \frac{v-4 +  w-2}{6} + 4 \cdot [0, \frac{\min\{v-4, w-2\}}{6}] \\ & = 3+ \frac{v +  w}{6} + 4 \cdot [0, \frac{\min\{v-4, w-2\}}{6}]
\end{aligned}
\]
and
\[\mathsf L ( g^{v+4} (-g)^{w}) =  2 + \frac{v +w}{6} + 4 \cdot [0, \frac{\min\{v+2, w-2\}}{6}].\]
Thus, the union is an AMP with difference $4$ and period $\{0,1,4\}$.

\smallskip
\noindent
CASE 2: $\mathsf v_{2g} (A)=2$.

We set $A = (2g)^2 g^{v + 8}(-g)^{w+4}$ with $\mathsf v_g (A) = v+8 > 0$, $\mathsf v_{-g}(A) = w+4 > 0$, and $v, w \in \Z$. Then $v \equiv w + 4 \mod 6$. If $w$ is odd, then $\mathsf Z (A) = U_4 \mathsf Z (A')$ with $A' = U_4^{-1}A$. Thus we may suppose that $w$ is even.

First suppose that $\mathsf v_{-g} (A) < 4$. Then $\mathsf v_{-g} (A) =2$, $A = (-g)^2 (2g)^2 g^{6m +4}$ for some $m \in \N_0$, and $\mathsf L (A) = \{m+2, m+3\}$ is an AP with difference $1$. Now suppose that $\mathsf v_{-g} (A) \ge 4$. We discuss the cases where $\mathsf v_{g} (A) \in \{2,4,6\}$. If $\mathsf v_g (A) = 2$, then $A = g^2 (2g)^2 (-g)^{6m}$ for some $m \in \N$ whence $\mathsf L (A) = \{m+1, m+2\}$. If $\mathsf v_g (A) = 4$, then $A = g^2 (2g)^2 (-g)^{6m+2}$ for some $m \in \N$,
\[
A = U_2 V_1 (-U_1)^m = U_3U_4^2 (-U_1)^m = V_1^2U_4^4 (-U_1)^{m-1}
\]
whence $\mathsf L (A) = \{2+m,3+m,5+m\}$ is an AMP with period $\{0,1,3,4\}$. If $\mathsf v_g (A) = 6$, then $A = g^6 (2g)^2 (-g)^{6m+4}$ for some $m \in \N_0$,
\[
A = U_3U_4^4 (-U_1)^m = U_2V_1 U_4^2 (-U_1)^m = V_1^2 U_1 (-U_1)^m = V_1^2U_4^6 (-U_1)^{m-1} \,,
\]
where the last equations holds only in case $m \ge 1$. Thus $\mathsf L (A) = \{3,4,5\}$ or $\mathsf L (A) = \{3+m,4+m,5+m,7+m\}$ if $m \in \N$.

Now we suppose that $\mathsf v_{-g} (A) \ge 4$ and $\mathsf v_{g} (A) \ge 8$ whence $A = (2g)^2 g^{v + 8}(-g)^{w+4}$ with $v, w \in \N_0$ even. Then
\[
\begin{aligned}
A & = U_2^2 g^v (-g)^{w+4} = U_2^2A_1
   = U_3 g^{v+6}(-g)^{w+4} = U_3A_2 \\
  & = V_1^2 g^{v+8}(-g)^w = V_1^2 A_3
   = V_1U_2 g^{v+4}(-g)^{w+2} = V_1U_2 A_4
\end{aligned}
\]
whence
\[
\begin{aligned}
\mathsf L (A) & = 2 + \Big( \mathsf L (A_1) \cup \big(-1+\mathsf L (A_2) \big) \cup \mathsf L (A_3) \cup \mathsf L (A_4) \Big) \\
     & = 2 + \Big(  \big(-1+\mathsf L (A_2) \big) \cup \mathsf L (A_3) \cup \mathsf L (A_4) \Big) \,,
\end{aligned}
\]
where the last equation holds because  $1+\mathsf L (A_1) \subset \mathsf L ( g^6 A_1)$ and $g^6A_1 = A_2$. The sets $\mathsf L (A_2), \mathsf L (A_3)$ and $\mathsf L (A_4)$ are APs with difference $4$. Thus in order to show that $\mathsf L (A)$ is an AMP with difference $4$, we study the minima and the lengths of $-1+\mathsf L (A_2), \mathsf L (A_3)$ and $\mathsf L (A_4)$. We distinguish three cases depending on the congruence class of $v$ modulo $6$ (recall that $v$ is even).

Suppose that $v \equiv 0 \mod 6$. The set $-1+\mathsf L (A_2)$ has minimum $a_0 = (v+w+4)/6$ amd length $\ell_0 = \min \{v+6, w+4\}/6$. The set $\mathsf L (A_3)$ has minimum $a_0+2$ and length $\min \{v+6, w-2\}/6 \in \{\ell_0, \ell_0 - 1\}$. Finally $\mathsf L (A_4)$ has minimum $a_0+3$ and length $\ell_0 - 1$. This implies that $\mathsf L (A)$ is an AMP with period $\{0,2,3,4\}$.

Suppose that $v \equiv 2 \mod 6$. Then $\mathsf L (A_4)$ has minimum $a_2= (v+w+6)/6$ and length $\ell_2=\min \{v+4, w+2\}/6$. The set $-1 + \mathsf L (A_2)$ has minimum $1+a_2$ and length $\ell_2$. The set $\mathsf L (A_3)$ has minimum $a_2+3$ and length $\ell_2$ or $\ell_2-1$. This implies that $\mathsf L (A)$ is an AMP with period $\{0,1,3,4\}$.

Suppose that $v \equiv 4 \mod 6$. The set $\mathsf L (A_3) $ has minimum $a_4= (v+w+8)/6$ and length $\ell_4 = \min \{v+8,w\}/6$. The set $\mathsf L (A_4)$ has minimum $a_4+1$ and length $\ell_4$ or $\ell_4-1$. The set $-1+\mathsf L (A_2)$ has minimum $a_4+2$ and the same length as $\mathsf L (A_4)$. This implies that $\mathsf L (A)$ is an AMP with period $\{0,1,2,4\}$.

\smallskip
\noindent
CASE 3: $\mathsf v_{2g} (A) \ge 3$.

We assert that $\mathsf L (A)$ is an AP with difference $1$.
We choose a factorization $z \in \mathsf Z (A)$ and we show that either $|z|= \max \mathsf L (A)$ or that $|z|+1 \in \mathsf L (A)$. Since $V_2U_5 = V_1^3$, we may suppose that $\min \{\mathsf v_{V_2}(z), \mathsf v_{U_5}(z) \} = 0$.

Suppose that $V_1 \mid z$. If $U_i \mid z$ for some $i \in [1,3]$, then $|z|+1 \in \mathsf L (A)$. If $\mathsf v_{U_5}(z)=0$, then $|z|=\max \mathsf L (A)$. Suppose that $\mathsf v_{V_2} (z)=0$. Then $z= U_4^{k_4}U_5^{k_5}V_1^{\ell_1}$ with $k_4, k_5, \ell_1 \in \N_0$,  and we claim that then $|z|=\max \mathsf L (A)$. By Lemma \ref{wichtig-0}, it suffices to show that $k_5+\ell_1 = \max \mathsf L (U_5^{k_5}V_1^{\ell_1})$. Since $U_5^{k_5}V_1^{\ell_1}$ is not divisible by an atom of length two, it follows that $\max \mathsf L (U_5^{k_5}V_1^{\ell_1}) \le |U_5^{k_5}V_1^{\ell_1}|/3 = k_5+\ell_1$.

Suppose that $V_1 \nmid z$. If $\mathsf v_{V_2} (z) = 0$, then $|z| = \max \mathsf L (A)$ because all remaining atoms have $g$-norm one. From now on we suppose that $V_2 \mid z$.
This implies that  $U_5 \nmid z$. Since $\mathsf v_{2g} (A) \ge 3$, then $U_5 \mid U_2^{\mathsf v_{U_2}(z)}U_3^{\mathsf v_{U_3}(z)}$ and we obtain a factorization $z' \in \mathsf Z (A)$ with $|z'|=|z|$ and with $U_5 \mid z'$ and we still have that $V_2 \mid z'$. Since $V_2U_5 = V_1^3$, it follows that $|z'|+1 \in \mathsf L (A)$.
\end{proof}

\begin{lemma} \label{wichtig-4}
Let $G$ be a cyclic group of order $|G|=6$, $g \in G$ with $\ord (g)=6$, and $A \in \mathcal B (G)$ with $\supp (A) = \{g, 2g,4g,  -g\}$. If $\mathsf v_{2g}(A) + \mathsf v_{4g}(A) \ge 3$, then  $\mathsf L (A)$ is an {\rm AP} with difference $1$ and otherwise it is an {\rm AMP} with period $\{0,1,2,4\}$ or $\{0,1,3,4\}$ or $\{0,2,3,4\}$.
\end{lemma}

\begin{proof}
We set $G_0 = \{g, 2g, 4g, -g\}$ and list all atoms of $\mathcal B (G_0)$:
\begin{itemize}
\item[] $\| \cdot  \|_g = 1:$ $U_0=g^6$, $U_1=g^4 (2g)$,  $U_2=g^2(2g)^2$,   $U_3=(2g)^3$, $V_2=(2g)(4g)$, \\ \phantom{$\| \cdot  \|_g = 1:$} $W= (4g)g^2$, $V_1=g(-g)$,

\item[] $\| \cdot  \|_g = 2:$  $-U_3=(4g)^3$, $-W=(2g)(-g)^2$,

\item[] $\| \cdot  \|_g = 3:$ $-U_2=(4g)^2(-g)^2$,

\item[] $\| \cdot  \|_g = 4:$ $-U_1=(-g)^4(4g)$,

\item[] $\| \cdot  \|_g = 5:$ $-U_0=(-g)^6$
\end{itemize}

First, suppose that $\mathsf v_{2g}(A) + \mathsf v_{4g}(A) \ge 3$. We choose a factorization $z \in \mathsf Z (A)$ and we show that either $|z|= \max \mathsf L (A)$ or that $|z|+1 \in \mathsf L (A)$. We write $z$ in the form
\[
z= z^+z^0z^- \,,
\]
where $z^+$ is the  product of all atoms from $\mathcal B (\{g, 2g\})$ (these are $U_0, U_1,U_2,U_3$), $z^-$ is the product of all atoms from $\mathcal B ( \{-g, 4g\})$, and $z^0$ is the product of all remaining atoms. Note that the sets in $\mathcal L(\{g, 2g\})$ and $\mathcal L (\{-g, 4g\})$ are singletons. If $U_3 \mid z^+$ and $z^- \ne 1$, then $|z|+1 \in \mathsf L (A)$. Similarly, if $-U_3 \mid z^-$ and $z^+ \ne 1$, then $|z|+1 \in \mathsf L (A)$.

Suppose that $z^-=1$. If $-W \nmid z$, then all atoms dividing $z$ have $g$-norm equal to $1$ whence $|z| = \max \mathsf L (A)$. Suppose that $-W \mid z$. If $(-W)Y \mid z$ with $Y \in \{U_0, U_1, U_2, W\}$, then $|z|+1 \in \mathsf L (A)$. Otherwise, $z$ is a product of the atoms $-W, U_3, V_2$, and $V_1$ which implies that $|z|= \max \mathsf L (A)$. The case $z^+=1$ follows by symmetry.

Thus we may suppose that $z^+ \ne 1$ and $z^+\ne 1$ and that $U_3 \nmid z^+$ and that $(-U_3) \nmid z^-$. Let $A^+$ resp. $A^-$ denote the zero-sum sequences corresponding to $z^+$ resp. $z^-$. If $\mathsf v_{2g}( A^+) \ge 3$, then there is a factorization $z'$ of $A^+$ with $|z'|=|z^+|$ and with $U_3 \mid z'$ and we are back to an earlier case. Thus we may suppose that $\mathsf v_{2g} ( A^+) \le 2$. By symmetry we may also suppose that $\mathsf v_{4g}(A^-) \le 2$.

Since $(-U_1)^2 = (-U_0)(-U_2)$, we may suppose that $\mathsf v_{-U_1}(z^-) \le 1$. By symmetry we infer that $\mathsf v_{U_1} (z^+) \le 1$. Now we distinguish two cases.

\noindent
CASE 1: \ $U_2 \mid z^+$ or $(-U_2) \mid z^-$.

By symmetry we may suppose that $U_2 \mid z^+$. If $(-U_2) \mid z^-$ or $(-U_1) \mid z^-$, then $|z|+1 \in \mathsf L (A)$. Otherwise $z^-$ is  a product of $(-U_0)$. Since $\mathsf v_{4g}(A) \ge 1$, it follows that $V_2 \mid z$ or $W \mid z$. Since $(-U_0)V_2= (-W)(-U_1)$, we are back to a previous case. Since $(-U_0)W = V_1^2 (-U_1)$, it follows that $|z|+1 \in \mathsf L (A)$.

\noindent
CASE 2: \ $U_2 \nmid z^+$ and $(-U_2) \nmid z^-$.

If $(-W) \mid z_0$, then $|z|+1 \in \mathsf L (A)$ because $z^+$ is a nonempty product of $U_0$ and $U_1$. If $W \mid z_0$, then $|z|+1 \in \mathsf L (A)$ because $z^-$ is a nonempty product of $-U_0$ and $-U_1$. From now on we suppose that $W \nmid z_0$ and $-W \nmid z_0$. Since $\mathsf v_{2g}(A) + \mathsf v_{4g}(A) \ge 3$, it follows that $V_2 \mid z_0$. If $U_0 \mid z$, then $U_0V_2= W U_1$ leads back to the case just handled. If $U_1 \mid z$, then $U_1V_2 = U_2W$ and we are back to CASE 1.

\smallskip
Now suppose that $\mathsf v_{2g}(A)=\mathsf v_{4g} (A)=1$. Then $\mathsf v_g (A) \equiv \mathsf v_{-g} (A) \mod 6$, and we set $\mathsf v_g (A) = v=6m+r$ and $\mathsf v_{-g} (A) = w = 6n+r$ with $v,w \in \N, r \in [0,5]$, and $m,n \in \N_0$. Then
\[
\begin{aligned}
A & = V_2 g^v (-g)^w \\
  & = U_1W g^{v-6}(-g)^w \\
  & = (-U_1)(-W) g^v (-g)^{w-6} \\
  & = W (-W) g^{v-2}(-g)^{w-2} \\
  & = U_1 (-U_1) g^{v-4}(-g)^{w-4} \,,
\end{aligned}
\]
where negative exponents are interpreted in the quotient group of $\mathcal F (G)$.
Since each factorization of $A$ is divisible by exactly one of the following
\[
V_2, U_1W, (-U_1)(-W), W(-W), U_1(-U_1) \,,
\]
we infer that $\mathsf L (A) = L_1 \cup L_2 \cup L_3 \cup L_4 \cup L_5$, where
\[
\begin{aligned}
L_1 & = 1 + r + (m+n) + 4 \cdot [0, \min \{m,n\}] \\
L_2 & = 2 + r + (m+n-1) + 4 \cdot [0, \min \{m-1,n\}] \\
L_3 & = 2 + r + (m+n-1) + 4 \cdot [0, \min \{m,n-1\}] \\
L_4 & = \begin{cases}
        2 + (r-2) + (m+n) + 4 \cdot [0, \min \{m,n\}] & r \in [2,5] \\
        2 + (r-2+6) + (m+n-2) + 4 \cdot [0, \min \{m-1,n-1\}] & r \in [0,1]
        \end{cases} \\
L_5 & = \begin{cases}
        2 + (r-4) + (m+n) + 4 \cdot [0, \min \{m,n\}] & r \in [4,5] \\
        2 + (r-4+6) + (m+n-2) + 4 \cdot [0, \min \{m-1,n-1\}] & r \in [0,3]
        \end{cases} \\
\end{aligned}
\]
Note, if exponents in the above equations are negative, then the associated sets $L_i$ are indeed empty. More precisely, if
\begin{itemize}
\item $v < 6$, then $m=0$, $[0, \min \{m-1,n\}]=\emptyset$, and hence $L_2=\emptyset$;
\item $w < 6$, then $n=0$ and $L_3=\emptyset$;
\item $r \in [0,1]$ and $\min \{m,n\}=0$, then $L_4=\emptyset$;
\item $r \in [0,3]$ and $\min \{m,n\}=0$, then $L_5=\emptyset$.
\end{itemize}
We observe that $L_2 \subset L_1$ and $L_3 \subset L_1$.

If $r \in [4,5]$, then $\min L_5 = r-2+m+n$, $\min L_4 = 2+\min L_5$, and $\min L_1=3+\min L_5$; since $L_5$, $L_4$, and $L_1$ are APs with difference $4$ and length $\min \{m,n\}$, $\mathsf L (A)$ is an AMP with period $\{0,2,3,4\}$.

If $r \in [2,3]$, then $\min L_4 = r+m+n$, $\min L_1=1+\min L_4$, and $\min L_5=2+\min L_4$; since $L_4$, $L_1$, and $L_5$ are APs with difference $4$, $L_4$ and $L_1$ have length $\min \{m,n\}$, and $L_5$ has length $\min \{m,n\} -1$, $\mathsf L (A)$ is an AMP with period $\{0,1,2,4\}$.

If $r \in [0,1]$, then $\min L_1 = 1+r+m+n$, $\min L_5=1+\min L_1$, and $\min L_4=3+\min L_1$; since $L_1$, $L_5$, and $L_4$ are APs with difference $4$, $L_4$ and $L_5$ have length $\min \{m,n\}-1$, and $L_1$ has length $\min \{m,n\}$, $\mathsf L (A)$ is an AMP with period $\{0,1,3,4\}$.
\end{proof}

\begin{proposition} \label{3.3}
Let $G$ be a cyclic group  of order $|G|=  6$. Then every $L \in \mathcal L (G)$ has one of the following forms{\rm \,:}
\begin{itemize}
\item $L$ is an {\rm AP} with difference in $\{1,2,4\}$.

\item $L$ is an {\rm AMP} with one of the following periods: $\{0,1,4\}$, $\{0,3,4\}$, $\{0,1,2,4\}$, $\{0,1,3,4\}$ $\{0,2,3,4\}$.
\end{itemize}
Each of the mentioned forms is actually attained by arbitrarily large sets of lengths.
\end{proposition}

\begin{proof}
Let $A' \in \mathcal B (G)$. If $A' = 0^m A$ with $m \in \N_0$ and $A \in \mathcal B (G \setminus \{0\})$, then $\mathsf L (A') = m + \mathsf L (A)$. Thus it is sufficient to prove the assertion for $\mathsf L (A)$.
Let $g \in G$ with $\ord (g)=6$.

If $\supp (A)=G \setminus \{0\}$, then $\mathsf L (A)$ is an AP with difference $1$ by Proposition \ref{2.2}.4. Suppose that $\supp (A) \subset G \setminus \{g, -g\} = \{2g, 3g, 4g\} = G_0$. Since $(3g)^2$ is a prime element in $\mathcal B (G_0)$ and $\{0, 2g, 4g\}$ is a cyclic group of order three, it follows that $\mathsf L (A)$ is an AP with difference $1$ by Proposition \ref{3.1}. Thus from now on we may suppose that $g \in \supp (A)$. We distinguish two cases.

\smallskip
\noindent
CASE 1: \  $-g \notin \supp (A)$.

Then $\supp (A) \subset \{g,2g,3g,4g\}$ and we assert that $\mathsf L (A)$ is an AP with difference $1$. Since all atoms in $\mathcal B (\{g, 2g, 3g\}$ have $g$-norm one, the sets in $\mathcal L (\{g, 2g, 3g\})$ are singletons. Thus the claim holds if $\supp (A) \subset \{g, 2g, 3g\}$.  If $\supp (A) = \{g,2g,3g,4g\}$, then the claim follows from  Lemma \ref{wichtig-2}. Thus it remains to consider the cases where $\supp (A)$ equals one of the following three sets:
\[
\{g, 4g\}, \ \{g,2g,4g\} , \quad \text{or} \quad \{g,3g,4g\} \,.
\]
The multiplicity of $g$ in any zero-sum sequence $B \in \mathcal B ( \{g,2g,4g\}$ is even, whence $\mathcal L ( \{g, 2g, 4g\}) = \mathcal L ( \{2g, 4g\})$. Since $\{0, 2g, 4g\}$ is a cyclic group of order three, all sets of lengths in $\mathcal L ( \{0, 2g, 4g\}$ are APs  with difference $1$ by Proposition \ref{3.1}. If $\supp (A) = \{g, 3g, 4g\}$, then the claim follows from Lemma \ref{wichtig-2A}.

\smallskip
\noindent
CASE 2: \  $-g \in \supp (A)$.

Then $\{g, -g \} \subset \supp (A)$ and we distinguish four cases.

\smallskip
\noindent
CASE 2.1: \ $\supp (A) \subset G_0=\{g, 3g, -g\}$.

Since $\min \Delta (G_0) = 2$ by \cite[Lemma 6.8.5]{Ge-HK06a} and $\max \Delta (G_0)=4$, it follows that $\Delta (G_0)=\{2,4\}$. If $4 \in \Delta (\mathsf L (A))$, then $\mathsf L (A)$ is an AP with difference $4$ by \cite[Theorem 3.2]{Ch-Go-Pe14a}, and otherwise $\mathsf L (A)$ is an AP with difference $2$.

\smallskip
\noindent
CASE 2.2: \ $\supp (A) = \{g,  2g, -g\}$ or $\supp (A) = \{g, 4g, -g\}$.

By symmetry it suffices to consider the first case. Lemma \ref{wichtig-3} shows that $\mathsf L (A)$ has one of the given forms.

\smallskip
\noindent
CASE 2.3: \ $\supp (A) = \{g,  2g, 3g, -g\}$ or $\supp (A) = \{g,  3g, 4g, -g \}$.

By symmetry it suffices to consider the first case.  Lemma \ref{wichtig-1} shows that $\mathsf L (A)$ is an AP with difference $1$.

\smallskip
\noindent
CASE 2.4: \ $\supp (A) = \{g, 2g, 4g, -g \}$.

Lemma \ref{wichtig-4} shows that $\mathsf L (A)$ has one of the given forms.

\noindent
Finally, we note that the preceding lemmas imply that each of the mentioned forms is actually attained by arbitrarily large sets of lengths.
\end{proof}

The next proposition is used substantially in \cite[Subsection 4.1]{Ge-Sc-Zh17b}, where the system $\mathcal L (C_5)$ is written down explicitly.

\begin{proposition} \label{3.4}
Let $G$ be a cyclic group  of order $|G|=  5$. Then every $L \in \mathcal L (G)$ has one of the following forms{\rm \,:}
\begin{itemize}
\item $L$ is an {\rm AP} with difference in $\{1,3\}$.

\item $L$ is an {\rm AMP} with period $\{0,2,3\}$ or with period $\{0, 1, 3\}$.
\end{itemize}
Each of the mentioned forms is actually attained by arbitrarily large sets of lengths.
\end{proposition}

\begin{proof}
By Proposition \ref{2.3} we obtain that $\Delta^* (G) = \{1, 3\}$. Let $A' \in \mathcal B (G)$. If $A' = 0^m A$ with $m \in \N_0$ and $A \in \mathcal B (G \setminus \{0\})$, then $\mathsf L (A') = m + \mathsf L (A)$. Thus it is sufficient to prove the assertion for $\mathsf L (A)$. If $|\supp (A)|=1$, then $|\mathsf L (A)|=1$. If $|\supp (A)|=4$, then $\mathsf L (A)$ is an AP with difference $1$ by Proposition \ref{2.2}.4. Suppose that $|\supp (A)|=2$. Then there is a nonzero $g \in G$ such that $\supp (A) = \{ g, 2g\}$ or $\supp (A) = \{g, 4g\}$. If $\supp (A) = \{g, 2g\}$, then $\mathsf L (A)$ is an AP with difference $1$ (this can be checked directly by arguing with the $g$-norm). If $\supp (A) = \{g, 4g\}$, then $\mathsf L (A)$ is an AP with difference $3$.

Thus it remains to consider the case $|\supp (A)| = 3$. We set $G_0 = \supp (A)$. Then there is an element $g \in G_0$ such that $-g \in G_0$. Thus either $G_0 = \{g, 2g, -g\}$ or $G_0 = \{g, 3g, -g\}$. Since $\{g, 3g, -g\} = \{-g, 2(-g), -(-g)\}$, we may suppose without restriction that $G_0 = \{g, 2g, -g\}$.
We list all atoms of $\mathcal B (G_0)$:
\begin{itemize}
\item[] $\| \cdot  \|_g = 1:$ $g^5$, $g^3 (2g)$, $g(2g)^2$, $g(-g)$,

\item[] $\| \cdot  \|_g = 2:$ $(2g)^5$, $(2g)^3(-g)$, $(2g)(-g)^2$,

\item[] $\| \cdot  \|_g = 4:$ $(-g)^5$.
\end{itemize}

If $\Delta ( \mathsf L (A) ) \subset \{1\}$, then $\mathsf L (A)$ is an AP with difference $1$. If $3 \in \Delta ( \mathsf L (A) )$, then $\Delta ( \mathsf L (A) ) = \{3\}$ by \cite[Theorem 3.2]{Ch-Go-Pe14a}, which means that $\mathsf L (A)$ is an AP with difference $3$. Thus it remains to consider the case where $2 \in \Delta ( \mathsf L (A) ) \subset [1,2]$.
We show that $\mathsf L (A)$ is an {\rm AMP} with period $\{0,2,3\}$ or with period $\{0, 1, 3\}$. Since $2 \in \Delta ( \mathsf L (A) )$, there exist $k \in \N$, $A_1, \ldots, A_k, B_1 , \ldots, B_{k+2} \in \mathcal  A (G_0)$ such that
\[
A = A_1 \cdot \ldots \cdot A_k = B_1 \cdot \ldots \cdot B_{k+2} , \quad \text{and} \quad k+1 \notin \mathsf L (A) \,.
\]
We distinguish two cases.

\smallskip
\noindent
CASE 1: \ $(-g)^5 \notin \{A_1, \ldots, A_k\}$.

Then $\{A_1, \ldots, A_k\}$ must contain atoms with $g$-norm $2$. These are the atoms $(2g)^5, (2g)(-g)^2, (2g)^3(-g)$. If $g^5$ or $g^3(2g)$ occurs in $\{A_1, \ldots, A_k\}$, then $k+1 \in \mathsf L (A)$, a contradiction. Thus none of the elements $(-g)^5, g^5$, and $g^3(2g)$ lies in $\{A_1, \ldots, A_k\}$, and hence
\[
\{A_1, \ldots, A_k\} \subset \{ (2g)^5, (2g)^3(-g), g(2g)^2, (2g)(-g)^2, g(-g) \} \,.
\]
Now we set $h = 2g$ and obtain that
\[
\begin{aligned}
\{A_1, \ldots, A_k\} & \subset \{ (2g)^5, (2g)^3(-g), g(2g)^2, (2g)(-g)^2, g(-g) \} \\ & = \{ h^5, h^3(2h), h^2(3h), h(2h)^2, (2h)(3h)  \}\,.
\end{aligned}
\]
Since the $h$-norm of all these elements equals $1$, it follows that $\max \mathsf L (A) = k$, a contradiction.

\smallskip
\noindent
CASE 2: \ $(-g)^5 \in \{A_1, \ldots, A_k\}$.

If $(2g)^5$, or $g(2g)^2$, or $(2g)^3(-g)$ occurs in $\{A_1, \ldots, A_k\}$, then $k+1 \in \mathsf L (A)$, a contradiction. Since $\Delta ( \{-g, g\}) = \{3\}$, it follows that
\[
\Omega = \{A_1, \ldots, A_k\} \cap \{ g^3(2g), (2g)(-g)^2\} \ne \emptyset \,.
\]
Since $\Big( g^3(2g) \Big) \Big( (2g)(-g)^2 \Big) = \Big( (-g)g \Big)^2 \Big( g (2g)^2 \Big)$ and $k+1 \notin \mathsf L (A)$, it follows that $|\Omega| = 1$. We distinguish two cases.

\smallskip
\noindent
CASE 2.1: \ $\{A_1, \ldots, A_k\} \subset \{ g^5, (-g)^5, g(-g), (2g)(-g)^2 \}$.

We set $h = -g$, and observe that
\[
\{A_1, \ldots, A_k\} \subset \{ g^5, (-g)^5, g(-g), (2g)(-g)^2 \} = \{ h^5, (-h)^5, h(-h), h^2(3h) \} \,.
\]
Since $(-h)^5$ is the only element with $h$-norm greater than $1$, it follows that $(-h)^5 \in \{A_1, \ldots, A_k\}$. Since $\Delta ( \{h, -h\}) = \{3\}$, it follows that $h^2(3h) \in \{A_1, \ldots, A_k\}$. Since $\Big( (-h)^5 \Big) \Big( h^2(3h) \Big) = \Big( h(-h) \Big)^2 \Big( (3h)(-h)^3 \Big)$, we obtain that $k+1 \in \mathsf L (A)$, a contradiction.

\smallskip
\noindent
CASE 2.2: \ $\{A_1, \ldots, A_k\} \subset \{ g^5, (-g)^5, g(-g),  g^3(2g)  \} $.

Since $\Big( g^3 (2g) \Big)^2 \Big((-g)^5\Big) = \Big( g^5 \Big) \Big( g(-g) \Big) \Big( (2g)(-g)^2 \Big)^2$ and $k+1 \notin \mathsf L (A)$, it follows that
\[
| \{i \in [1,k] \mid A_i = g^3 (2g) \}| = 1 \,,
\]
and hence $\mathsf v_{2g} (A) = 1$. Thus every factorization $z$ of $A$ has the form
\[
z = \big( (2g)g^3\big) z_1 \quad \text{or } \quad z= \big( (2g)(-g)^2 \big) z_2 \,,
\]
where $z_1, z_2$ are factorizations of elements $B_1, B_2 \in \mathcal B ( \{-g,g\})$. Since $\mathsf L (B_1)$ and $\mathsf L (B_2)$ are APs of difference $3$, $\mathsf L (A)$ is a union of two shifted APs of difference $3$. We set
\[
A = \big(g^5 \big)^{m_1} \big( (-g)^5 \big) \big( (-g)g \big)^{m_3} \big( (2g)g^3 \big) \,,
\]
where $m_1 \in \N_0$, $m_2 \in \N$, and $m_3 \in [0,4]$. Suppose that $m_1 \ge 1$. Note that
\[
A'= \big(g^5\big) \big( (-g)^5\big) \big( (2g)g^3\big) = \big( (-g)g\big)^3 \big( (2g)(-g)^2\big) \big(g^5\big) = \big( (-g)g\big)^5 \big( (2g)g^3 \big) \,,
\]
and hence $\mathsf L (A')=\{3,5,6\}$. We set $A=A' A''$ with $A'' \in \mathcal B (\{g,-g\})$. The above argument on the structure of the factorizations of $A$ implies that $\mathsf L (A)$ is the sumset of $\mathsf L (A')$ and $\mathsf L (A'')$ whence
\[
\mathsf L (A)=\mathsf L (A')+\mathsf L (A'') = 3 + \{0,2,3\} + \mathsf L (A'') \,.
\]
Since $\mathsf L (A'')$ is an AP with difference $3$, $\mathsf L (A)$ is an AMP with period $\{0,2,3\}$. Suppose that $m_1=0$. If $m_3 \in [2,4]$, then $\mathsf L (A)= \{m_2+m_3,m_2+m_3+1,m_2+m_3+3\}$ is an AMP with period $\{0,1,3\}$.  If $m_3=1$, then $\mathsf L (A)= \{m_2+2,m_2+4\}$. If $m_3=0$, then $\mathsf L (A)=\{m_2+1, m_2+3\}$.
\end{proof}

\begin{proposition} \label{3.5}
Let $G$ be a finite abelian group, $e_1, e_2 \in G \setminus \{0\}$ with $\langle e_1 \rangle \cap \langle e_2 \rangle = \{0\}$, $\ord (e_2)=n \ge 4$, and $U = (e_1+e_2)(-e_1)e_2^{n-1}$. Then, for every $k \in \N$, we have
\[
\mathsf L \Big( U (-U) \big(e_2 (-e_2) \big)^{kn} \Big) = \{2,n,n+1\} + 2k + (n-2) \cdot [0,k] \,.
\]
\end{proposition}

\begin{proof}
Let $k \in \N$ and $A_k =  U (-U) \big(e_2 (-e_2) \big)^{kn} $. If $X \in \mathcal A (G)$ with $X \mid A_k$, then
\[
\begin{aligned}
X \in \{ & U, -U, V = (e_1+e_2)(-e_1)(-e_2), -V, W_0 = (e_1+e_2)(-e_1-e_2), \\
   &  W_1 = e_1(-e_1), W_2 = e_2(-e_2), e_2^n, (-e_2)^n \} \,.
\end{aligned}
\]
Thus
\[
\begin{aligned}
\mathsf Z (A_k)  = & W_0W_1 \mathsf Z (W_2^{kn+n-1}) \ \cup \\
 & U(-U) \mathsf Z (W_2^{kn}) \ \cup  \ V(-V)\mathsf Z (W_2^{kn+(n-2)} \ \cup \\
  & U(-V) \mathsf Z ( W_2^{kn-1} (-e_2)^n) \ \cup \ (-U)V \mathsf Z (W_2^{kn-1} e_2^n) \,,
\end{aligned}
\]
and we obtain the claim by using \eqref{basic}.
\end{proof}

\begin{proposition} \label{3.6}
Let $G = C_2^4$, $(e_1, e_2, e_3, e_4)$ be a basis of $G$, $e_0=e_1+\ldots+e_4$,  $U_4=e_0 \cdot \ldots \cdot e_4$, $U_3 = e_1e_2e_3(e_1+e_2+e_3)$, and $U_2= e_1e_2(e_1+e_2)$.
\begin{enumerate}
\item For each $k \in \N$, $\mathsf L \big( (U_3U_4)^{2k} \big) = \{4k\} \cup [4k+2, 9k]$.

\item For each $k \in \N$, $\mathsf L (U_4^{2k}U_2)= (2k+1)+\{0,1,3\} + 3 \cdot [0, k-1]$.
\end{enumerate}
\end{proposition}

\begin{proof}
1.  Let $k \in \N$ and $A_k = (U_3U_4)^{2k}$. Then
\[
\begin{aligned}
\{4k\} \cup [4k+2, 9k-2] \cup \{9k\} & =  \big( 2k + 2 \cdot [0,k] \big) + \big( 2k + 3 \cdot [0,k] \big) = \\
\mathsf L (U_3^{2k}) + \mathsf L (U_4^{2k}) & \subset \mathsf L \big( (U_3U_4)^{2k} \big) \,.
\end{aligned}
\]
Since $\max \mathsf L (A_k) \le |A_k|/2 = 9k$, it follows that $\max \mathsf L (A_k)=9k$. We set $z = U_3^{2k}U_4^{2k} \in \mathsf Z (A_k)$. Since $U_4$ is the only atom of length $\mathsf D (G)=5$ dividing $A_k$ and since $2k=\mathsf v_{e_0}(A_k) = \mathsf v_{U_4} (z) = \max \{\mathsf v_{U_4}(z') \mid z' \in \mathsf Z (A_k)\}$, it follows that $\min \mathsf L (A_k)=|z|=4k$.

Next we assert that $4k+1 \notin \mathsf L (A_k)$. For $\nu \in [0,4]$, we set $V_{\nu} = e_{\nu}^2$ and $V_5 = (e_1+e_2+e_3)^2$. Since $\mathsf Z (U_3^2) = \{U_3^2, V_1V_2V_3V_5\}$, $\mathsf Z (U_4^2) = \{ U_4^2, V_1V_2V_3V_4V_0\}$, and $\mathsf Z (U_3U_4) = \{U_3U_4, V_1V_2V_3W\}$ where $W = (e_1+e_2+e_3) e_0 e_4$, it follows that $\min ( \mathsf L (A_k) \setminus \{4k\}) = 4k+2$.

Since $A_k$ has precisely one factorization $z$ of length $|z|=\max \mathsf L (A_k)=9k$, and $z$ is a product of  atoms of length two with $V_0V_4V_5 \mid z$. Since $V_0V_4V_5=W^2$, it follows that $A_k$ has a factorization of length $9k-1$.

\smallskip
2. Setting $W=(e_1+e_2)e_3e_4e_0$ we infer that $U_4^2U_2= U_4 (e_1^2)(e_2^2)W = U_2 (e_0^2) \cdot \ldots \cdot (e_4^2)$ and hence $\mathsf L (U_4^2U_2)=\{3,4,6\}$. Thus for each $k \in \N$ we obtain that
\[
\begin{aligned}
\mathsf L (U_4^{2k}U_2) & = \big( \{1\}+\mathsf L (U_4^{2k}) \big) \cup \big( \mathsf L(U_4^{2k-2}) + \mathsf L (U_4^2U_2) \big) \\
 & = \big( 2k+1 + 3 \cdot [0,k] \big) \cup \big( 2k-2+ 3 \cdot [0, k-1] + \{3,4,6\} \big) \\
 & = \big( 2k+1 + 3 \cdot [0,k] \big) \cup \big( 2k+2 + 3 \cdot [0,k-1] \big) \cup \big( 2k+4 + 3 \cdot [0,k-1] \big) \\
 & = (2k+1) + \{0,1,3\} + 3 \cdot [0,k-1] \,. \qedhere
\end{aligned}
\]
\end{proof}

\begin{proposition} \label{3.7}
Let $G = C_3^r$ with $r \in [3,4]$, $(e_1, \ldots ,e_r)$  a basis of $G$, $e_0=e_1 + \ldots + e_r$, and $U=(e_1 \cdot \ldots \cdot e_r)^2e_0$.
\begin{enumerate}
\item If $r=3$, then for each $k \in \N$ we have 
      \[
      \mathsf L \big(  U^{6k+1} (-U) \big) = [6k+2, 14k+5] \cup \{14k+7\} \,.
      \]

\item If $r=4$ and $V_1=e_1^2e_2^2(e_1+e_2)$, then for each $k \in \N$ we have
      \[
      \mathsf L (U^{3k}V_1)= (3k+1)+\{0,1,3\} + 3 \cdot [0, 2k-1] \,.
      \]
\end{enumerate}
\end{proposition}

\begin{proof}
1. Let $r=3$ and $k \in \N$. We set $A_k = U^{6k+1} (-U)$ and $L_k = \mathsf L (A_k)$. For $\nu \in [0,3]$, we set $U_{\nu} = e_{\nu}^3$, $V_{\nu} = (-e_{\nu})e_{\nu}$, and we define $X  = e_0^2e_1e_2e_3$.

First, consider $\mathsf L ( U^{6k} )$.
We observe that $\mathsf Z (U^2) = \{U^2, U_1U_2U_3 X \}$ and   $\mathsf Z (U^3)= \{U^3, UU_1U_2U_3 X,$ $U_0U_1^2U_2^2U_3^2\}$. Furthermore,  $\min \mathsf L ( U^{6k}) = 6k$, $\max \mathsf L ( U^{6k}) = 14k$, $\Delta ( \{e_0, \ldots, e_3\})=\{2\}$, and hence
\[
\mathsf L ( U^{6k} ) = 6k +  2 \cdot [0,4k] \,.
\]
Next, consider $\mathsf L \big( (-U)U \big)$. For subsets $I, J \subset [1,3]$ with  $[1,3] = I \uplus J$, we set
\[
W_I = e_0 \prod_{i \in I} e_i^2 \prod_{j \in J} (-e_j) \,.
\]
Since
\[
\mathsf Z \big( U(-U) \big) = \Big\{ V_0V_1^2V_2^2V_3^2 \Big\} \ \uplus \ \Big\{  W_I (-W_I) \prod_{j \in J} V_j \mid I, J \subset [1,3] \ \text{ with} \ [1,3] = I \uplus J \Big\} \,,
\]
it follows that
\[
\mathsf L \big( (-U)U \big) = \Big\{ 7 \Big\} \ \uplus \ \Big\{2+|J| \mid I, J \subset [1,3] \ \text{ with} \ [1,3] = I \uplus J \Big\} =  \{2,3,4,5,7\} \,,
\]
which implies that
\[
[6k+2, 14k+5] \cup \{14k+7\} = \mathsf L \big( (-U)U \big) + \mathsf L ( U^{6k} ) \subset \mathsf L (A_k) \subset [6k+2, 14k+7] \,.
\]
It remains to show that $14k+6 \notin L_k$. For this, we consider the unique factorization $z \in \mathsf Z (A_k)$ of length $|z|=14k+7$ which has the form
\[
z = \big( U_0 U_1^2 U_2^2 U_3^2)^{2k}  \big( V_0 V_1^2V_2^2V_3^2 \big)  \,.
\]
Assume to the contrary that there is a factorization $z' \in \mathsf Z (A_k)$ of length $|z'| = 14k+6$. If $V_0 \t z'$, then $V_0 V_1^2V_2^2V_3^2 \t z'$ and $z' = V_0 V_1^2V_2^2V_3^2 x$ with $x \in \mathsf Z (U^{6k})$. Whence $|x| \in \mathsf L (U^{6k})$ and $|z'| \in 7 + \mathsf L (U^{6k})$, a contradiction. Suppose that $V_0 \nmid z'$. Then there are $I, J \subset [1,3] \ \text{ with} \ [1,3] = I \uplus J$ such that $W_I (-W_I) \prod_{j \in J} V_j \t z'$ and hence $z' = W_I (-W_I) \big( \prod_{j \in J} V_j \big) x$ with $x \in \mathsf Z (U^{6k})$. Thus $|z'| \in [2,5] + \mathsf L (U^{6k})$, a contradiction.

\smallskip
2. Let $r=4$ and $k \in \N$. We have $\mathsf L (U^2)=\{2,5\}$ and $\mathsf L (U^{3k})=3k + 3 \cdot [0, 2k]$. We define
\[
V_2=(e_1+e_2)e_1e_2e_3^2e_4^2e_0 \,, \ V_3=(e_1+e_2)e_3e_4e_0^2, \quad \text{and} \quad W=e_1 \cdot \ldots \cdot e_4 e_0^2 \,,
\]
and observe that
\[
U^3V_1=U^2 V_2 (e_1^3)(e_2^3)=UV_3 (e_1^3)^2 (e_2^3)^2(e_3^3)(e_4^3)
\]
whence $\mathsf L (U^3V_1)= \{4,5,7,8\}$. Clearly, each factorization of $U^{3k}V_1$ contains exactly one of the atoms $V_1, V_2, V_3$, and it contains it exactly once. Therefore we obtain that
\[
\begin{aligned}
\mathsf L (U^{3k}V_1) & =  \big( \{1\} + \mathsf L (U^{3k}) \big) \cup \big( \mathsf L (U^3V_1) + \mathsf L (U^{3k-3}) \big)  \\
 & = \big( (3k+1) + 3 \cdot [0, 2k] \big) \cup \big( \{4,5,7,8\} + (3k-3) + 3 \cdot [0, 2k-2] \big) \\
 & = \big( (3k+1) + 3 \cdot [0, 2k] \big) \cup \big( (3k+1) + \{0,1,3,4\} + 3 \cdot [0, 2k-2] \big) \\
 & = (3k+1) + \{0,1,3\} + 3 \cdot [0, 2k-1] \,. \qedhere
\end{aligned}
\]
\end{proof}

\begin{proof}[Proof of Theorem \ref{1.1}]
1. (a)\, $\Rightarrow$\, (b) \ This implication is obvious.

(b)\, $\Rightarrow$\, (d) \ Suppose that $\exp (G) = n$, and that
$G$ is not isomorphic to any of the groups listed in (d). We have to show that there is
an $L \in \mathcal L (G)$ which is not an AP.
We distinguish four cases.

(d)\, $\Rightarrow$\, (a) \  Proposition \ref{3.1} shows that, for all groups mentioned,  all sets of lengths are APs. Proposition \ref{2.3} shows that all differences lie in $\Delta^* (G)$.

(c)\, $\Leftrightarrow$\, (d) \ This is the special case for finite groups of \cite[Theorem 1.1]{Ge-Sc16b} (see Remark \ref{3.8}.1).

\noindent
CASE 1: \ $n \ge 5$.

Then \cite[Proposition 3.6.1]{Ge-Sc16a} provides examples of sets of lengths which are not APs.

\noindent
CASE 2: \ $n = 4$.

Since $G$ is not cyclic, it has a subgroup isomorphic to $C_2 \oplus C_4$. Then \cite[Theorem 6.6.5]{Ge-HK06a} shows that $\{2,4,5\} \in \mathcal L (C_2 \oplus C_4) \subset \mathcal L (G)$.

\noindent
CASE 3: \ $n = 3$.

Then $G$ is isomorphic to $C_3^r$ with $r \ge 3$, and Proposition \ref{3.7}.1 provides examples of sets of lengths which are not APs.

\noindent
CASE 4: \ $n = 2$.

Then $G$ is isomorphic to $C_2^r$ with $r \ge 4$, and Proposition \ref{3.6}.1 provides examples of sets of lengths which are not APs.

\smallskip
2. (b)\, $\Rightarrow$\, (a) \ Suppose that $G$ is a subgroup of $C_4^3$ or a subgroup of $C_3^3$. Then Proposition \ref{2.3}.2 implies that $\Delta^* (G) \subset \{1,2\}$, and hence Proposition \ref{2.2}.1 implies the assertion.

(a)\, $\Rightarrow$\, (b) \ Suppose that (b) does not hold. Then $G$ has a subgroup isomorphic to a cyclic group of order $n \ge 5$, or isomorphic to $C_2^4$, or isomorphic to $C_3^4$.
We show that in none of these cases (a) holds.

If $G$ has a subgroup isomorphic to $C_n$ for some $n \ge 5$, then \cite[Proposition 3.6.1]{Ge-Sc16a} shows that (a) does not hold. If $G$ has a subgroup isomorphic to $C_2^4$, then Proposition \ref{3.6}.2 shows that (a) does not hold. If $G$ has a subgroup isomorphic to $C_3^4$, then Proposition \ref{3.7}.2 shows that (a) does not hold.

\smallskip
3. Let $\Delta$ be finite with $\Delta^* (G) \subset \Delta \subset \N$. If  $G$ is cyclic with $|G| \le 6$, then all sets of lengths are AMPs with difference in $\Delta^* (G)$ by  Propositions \ref{3.1}, \ref{3.3}, and \ref{3.4}. If $\exp (G) \ge 7$, then Proposition \ref{3.2} shows that  there are arbitrarily large sets of lengths that are not AMPs with difference $d \in \Delta$.

Suppose that $G$ has rank $r \ge 2$ and $\exp (G) =n  \in [2,6]$. If $n \ge 4$, then  Proposition \ref{3.5} shows that  there are arbitrarily large sets of lengths that are not AMPs with difference $d \in \Delta$.
Thus it suffices to consider elementary $2$-groups and elementary $3$-groups.

Suppose that $G = C_2^r$. If $r \le 3$, then the assertion follows from 1. If $r  \ge 4$, then the assertion follows from Proposition \ref{3.6}.1.

Suppose that $G = C_3^r$. If $r \le 2$, then the assertion follows from 1. If $r \ge 3$, then the assertion follows from Proposition \ref{3.7}.1.
\end{proof}

\begin{proof}[Proof of Corollary \ref{1.2}]
Let $G'$ be an abelian group such that $\mathcal L (G)=\mathcal L (G')$. Then $G'$ is finite by Proposition \ref{2.2} and by \cite[Theorem 7.4.1]{Ge-HK06a}.  By Proposition \ref{2.4}, we have $\mathsf D (G)=\rho_2 (G) = \rho_2 (G')=\mathsf D (G')$, and $\mathcal L (G)$ satisfies one of the properties given in Theorem \ref{1.1} if and only if the same is true for $\mathcal L (G')$. We distinguish three cases.

\smallskip
\noindent
CASE 1: \ $\mathcal L (G)$ satisfies the property in Theorem \ref{1.1}.1.

Then $G$ is cyclic of order $|G| \le 4$ or isomorphic to a subgroup of  $C_2^3$ or isomorphic to a subgroup of  $C_3^2$. Moreover, the same is true for $G'$.
Since $\mathsf D (G) \ge 4$, the assertion follows from Proposition \ref{3.1}.

\smallskip
\noindent
CASE 2: \ $\mathcal L (G)$ satisfies the property in Theorem \ref{1.1}.2.

By CASE 1, we may suppose that $\mathcal L (G)$ and  $\mathcal L (G')$ do not satisfy the property in Theorem \ref{1.1}.1.
Thus $G$ and $G'$, are isomorphic to one of the following groups:  $C_3^3$, $C_2 \oplus C_4$, $C_2^2\oplus C_4$, $C_2 \oplus C_4^2$,  $C_4^2$, or $C_4^3$. Since $C_3^3$ and $C_4^2$ are the only non-isomorphic groups having the same Davenport constant,  it remains to show that $\mathcal L (C_3^3) \ne \mathcal L (C_4^2)$. Since $\max \Delta (C_4^2)=3$ (by \cite[Lemma 3.3]{Ge-Zh15b}) and $\max \Delta (C_3^3)=2$ (by \cite[Proposition 5.5]{Ge-Gr-Sc11a}), the assertion follows.

\smallskip
\noindent
CASE 3: \ $\mathcal L (G)$ satisfies the property in Theorem \ref{1.1}.3.

By CASE 1, we may suppose that $G$ and  $G'$ do not satisfy the property in Theorem \ref{1.1}.1. Thus  $G$ and $G'$ are  cyclic of order five or six. Since $\mathsf D (C_5)=5$ and $\mathsf D (C_6)=6$, the claim follows.
\end{proof}

\begin{remark} \label{3.8}~

1. Let $H$ be an atomic monoid. The system of sets of lengths $\mathcal L (H)$ is said to be additively closed if the sumset $L_1+L_2 \in \mathcal L (H)$ for all $L_1, L_2 \in \mathcal L (H)$. Thus $\mathcal L (H)$ is additively closed if and only if $(\mathcal L (H), +)$ is a commutative reduced semigroup with respect to set addition. If this holds, then $\mathcal L (H)$ is an acyclic semigroup in the sense of Cilleruelo, Hamidoune, and Serra (\cite{Ci-Ha-Se10a}). $\mathcal L (H)$ is additively closed in certain Krull monoids stemming from module theory (\cite[Section 6.C]{Ba-Ge14b}). Examples in a non-cancellative setting can be found in \cite[Theorem 4.5]{Ge-Sc18b} and a more detailed discussion of the property of being additively closed is given in \cite{Ge-Sc16b}.

2. Several properties occurring in Theorem \ref{1.1} can be characterized by further arithmetical invariants such as  the catenary degree $\mathsf c (G)$ and the tame degree $\mathsf t (G)$ (for background see \cite[Sections 6.4 and 6.5]{Ge-HK06a}). For example, the properties (a) - (d) given in  Theorem \ref{1.1}.1. are equivalent to each of the following properties (e) and (f):
\begin{itemize}
\item[(e)] $\mathsf c (G) \le 3 \quad \text{or} \quad \mathsf c (G)=4 \ \text{and} \ \{2, 4\} \in \mathcal L (G)$.

\item[(f)] $\mathsf c (G) \le 3 \quad \text{or} \quad \mathsf t (G)=4$.
\end{itemize}
(use \cite[Theorem A]{Ge-Zh15b}, \cite[Theorem 4.12]{Ga-Ge-Sc15a}, and \cite[Theorem 6.6.3]{Ge-HK06a}).
\end{remark}


\providecommand{\bysame}{\leavevmode\hbox to3em{\hrulefill}\thinspace}
\providecommand{\MR}{\relax\ifhmode\unskip\space\fi MR }
\providecommand{\MRhref}[2]{%
  \href{http://www.ams.org/mathscinet-getitem?mr=#1}{#2}
}
\providecommand{\href}[2]{#2}

\end{document}